%% file: main.tex
\documentclass[12pt]{amsart}

\usepackage[margin=1in]{geometry}
\input{preamble.tex}
\usepackage{amsaddr}

\providecommand{\Keywords}[2]
{  \small
\noindent{\textit{\textbf{Keywords:}}} #1\\
  {\textit{\textbf{Math Subject Classification 2020:}}} #2
}
\title[Inhomogeneous Percolation on HCM]{Inhomogeneous percolation on the hierarchical configuration model with a heavy-tailed degree distribution}

\author{David Clancy, Jr.}
\address{Department of Mathematics, University of Wisconsin.\\
Email: \href{mailto:dclancy@math.wisc.edu}{dclancy@math.wisc.edu}.}
\date{\today\\
The author was partially supported by NSF DMS 2023239.}

\newcommand{\btheta}{\boldsymbol{\theta}}
\usepackage{verbatim}
\usepackage{todonotes}
\newcommand{\fl}[1]{\lfloor #1 \rfloor}
\begin{document}

\maketitle

\begin{abstract} 
We consider inhomogeneous percolation on a hierarchical configuration model with a heavy-tailed degree distribution. This graph is the configuration model where all the half-edges are colored either black or white, and edges are formed by uniformly matching edges of the same color. When only the white half-edges are paired, we provide sufficient conditions for the size and total number of incident black half-edges of the connected components to converge in an $\ell^2$-sense. The limiting vector is described by an $\R^2$-valued thinned L\'{e}vy process. We also establish an $\ell^2$-limit for the number of vertices in connected components when a critical proportion of the black edges are included. A key part of our analysis is establishing a Feller-type property for the multiplicative coalescent with mass and weight recently studied in \cite{DvdHvLS.17,DvdHvLS.20}.
\end{abstract}

{05C80,60C05}

\section{Introduction}

There has been a lot of interest in understanding the component structure of random graphs at and near the phase transition where a giant connected component emerges. A major breakthrough came when Aldous \cite{Aldous.97} introduced the multiplicative coalsecent to describe the dynamic merging of the connected components of the near-critical Erd\H{o}s-R\'{e}nyi random graph $G(n,n^{-1} + \lambda n^{-4/3})$ as the parameter $\lambda\in \R$ varies. In short, Aldous showed that in the $n\to\infty$ limit there is naturally emerging Feller process $\MC_1 = (\MC_1(\bx,t);t\ge 0)$ starting at $\MC_1(\bx,0) = \bx$ and taking values in 
\begin{equation*}
    \ell^2_\downarrow:= \left\{\bx = (x_1,x_2,\dotsm) \in \R_+^\infty : x_1\ge x_2\ge\dotsm\ge 0,\,\, \sum_{i} x_i^2 <\infty\right\}
\end{equation*} 
whose dynamics are described as follows. View $\bX$ as a vector of masses of countably many blocks listed from largest to smallest. Independently for all pairs of blocks $i$ and $j$ of respective masses $x_i$ and $x_j$, they merge into single block of mass $x_i+x_j$ at rate $x_ix_j$. The multiplicative coalescent has been quite useful in describing the component sizes of many near-critical random graph models; see \cite{Aldous.97,AL.98,BBSW.14,BDvdHS.20,vdHvLS.18,DvdHvLS.17,DvdHvLS.20} for an incomplete list of references.

It has recently been observed \cite{DvdHvLS.17,DvdHvLS.20} that critical bond percolation on the super-critical configuration model is better understood in the discrete using a generalization of the multiplicative coalescent where the merging rates depend on a ``hidden'' feature called the weight. More formally, \cite{DvdHvLS.17,DvdHvLS.20} introduce the multiplicative coalescent with mass and weight to describe the merging of blocks of various masses $(x_i)_{i\ge 1}$ and weights $(y_i)_{i\ge1}$ where independently for all pairs of blocks:
\begin{quote}
    blocks $i$ and $j$ merge into a single block of mass $x_i+x_j$ and weight $y_i+y_j$ at rate $y_iy_j$.
\end{quote} Following \cite{DvdHvLS.17,DvdHvLS.20}, we write $\MC_2(\bx,\by,t)$ for the ordered vector of component masses at time $t\ge 0$. Originally, the multiplicative coalescent with mass and weight (MCMW in the sequel) was used to describe the dynamics of near-critical percolation on the configuration model. Here the object of interest was the size of the connected components; however, the components coalesced at a rate depending on the number of ``open half-edges'' that were incident to vertices in the component. That is the size of the component was the mass, but the number of incident open half-edges was the weight. In the works mentioned above, asymptotically for each of the large components the weight was approximately a constant multiple of the mass. Consequently, the merging dynamics of MCMW were essentially the multiplicative coalescent run at a different speed: $\MC_2(\bx,c\bx,t ) = \MC_1(\bx,c^2t)$.

In this work, we are interested in inhomogeneous bond percolation on a hierarchical random graph model recently introduced in \cite{vdHvLS.18,SvdHvL.16,SvdHvL.16a} under the name hierarchical configuration model and in \cite{Kryven.19} under the name edge-colored configuration model. In its full generality, the hierarchical model exhibits arbitrary community structure on the mesoscopic scale, but on the macroscopic scale behaves like a configuration model. We show that critical bond percolation of the edges in the macroscopic configuration model is described asymptotically by the MCMW where the weight vector $\by$ has no direct relationship with the mass vector $\bx$. A key result in our analysis, which we believe will be useful for other inhomogeneous models of percolation, is the following Feller-type property. Write $\ell^2_+$ for the collection of square summable sequences of non-negative real numbers.
\begin{theorem}\label{thm:MC}
    Suppose that $\bx^n,\by^n\in \ell^2_+$ and $\bx^n\to \bx$ and $\by^n\to \by$ in $\ell^2$. If each coordinate of $\by$ is strictly positive, then for any fixed $t\ge 0$
    \begin{equation*}
        \MC_2(\bx^n,\by^n,t)\weakarrow \MC_2(\bx,\by,t)\qquad\textup{ in }\ell^2_\downarrow.
    \end{equation*}
\end{theorem}

\subsection{Description of Model}

Recently in \cite{SvdHvL.16,SvdHvL.16a, vdHvLS.18} Stegehuis, van der Hofstad and van Leeuwaarden have introduced and studied proprieties of the \textit{hierarchical configuration model.} The hierarchical configuration model (HCM in the sequel) is a random graph model which is constructed using the configuration model (which gives the macroscopic picture) and then by blowing up each vertex $v$ into a graph $F_v$. Formally, the HCM is a random graph constructed as follows. A \textit{community} $F = (V,E, \bd)$ is a connected multigraph $(V,E)$ on a finite number of vertices $V$ equipped with a $V$-indexed sequence $\bd = (d_v^{(b)};v\in V)$ of non-negative integers $d_v^{(b)} \in\{0,1,\dotsm\}$. We denote by $d^{(w)}_v = \deg_F(v)$, the degree of $v$ in the graph $F$. The superscript ``$(b)$'' stands for ``between'' and the superscript ``$(w)$'' stands for ``within.'' We view $d_v^{(b)}$ as the number of \textit{other} communities to which the vertex $v$ will be connected, although since we generate the HCM using the configuration model we can connect $v$ to another vertex in the same community. 

We will say that $d_v^{(b)}$ is the \textit{black degree} of the vertex $v$ and $d_v^{(w)}$ is the \textit{white degree} of the vertex $v$. We call the corresponding edges black and white as well. 

Given $n$ communities $F_1,F_2,\dotsm, F_n$, we let $V = \bigsqcup_{j\in[n]} V_j$. Let $G = (V,E_{\CM})$ the random graph formed using the configuration model on the vertex set $V$ with the degree sequence given by the black degrees $\bd^{(b)} = (d^{(b)}_v;v\in V)$. That is, each vertex $v\in V$ is incident to $d_{v}^{(b)}$ many half-edges and we uniformly at random pair these half-edges (and possibly deleting a single unmatched last half-edge). The HCM with communities $F = (F_j;j\in[n])$ is the graph on $V = \bigsqcup_{j\in[n]} V_j$ and edges $E = E_{\CM}\cup\bigcup_{j} E_j$ and we will denote this graph by $\HCM(F)$ if the communities are already given. 

In the above set-up the mescoscopic communities (formed by the white edges) are deterministic and the randomness comes into play by considering the macroscopic (black) edges. We will be examining the HCM where both the mesoscopic communities and the macroscopic graph are formed using the configuration model. More formally, let $\bd^{(w)} = (d_1^{(w)},\dotsm,d_n^{(w)})$ be a sequence of positive integers and let $\bd^{(b)} = (d_1^{(b)},\dotsm, d_n^{(b)})$ be a sequence of non-negative integers. We will assume that 
\begin{equation*}
    \ell^{(*)}_n := \sum_{i=1}^n d_i^{(\ast)} \text{ is even } \qquad \text{for }\ast\in \{b,w\}.
\end{equation*}
The graph $\HCM_n(\bd^{(w)},\bd^{(b)})$ is a random graph on $[n]$ constructed as follows. Start with $n$ vertices $1,2,\dotsm,n$ where vertex $i$ has $d_i^{(w)}$ many white half-edges connected to it and $d_i^{(b)}$ many black half-edges connected to it. Uniformly at random generate a matching of the $\ell^{(w)}$ many white half-edges and independently generate a uniform matching of the $\ell^{(b)}$ many black half-edges. For each pair of half-edges in each matching, turn the two half-edges into an edge in $\HCM_n(\bd^{(w)},\bd^{(b)})$.

We also note that this model also appears in \cite{Kryven.19} under the name \textit{edge-colored configuration model} and is one reason why we call the edges white and black. This is also an example of a graph model which is a \textit{network of networks} \cite{BD.14}. We prefer to stick with the name hierarchical configuration model.

We make the formation of the macroscopic graph dynamic. Attach an independent exponential rate $1$ clock to each black edge. We set $\G_n = (\G_n(s); s\ge 0)$
where $\G_n(s)$ includes all the white edges and only the black edges whose clock has rung before time $s$. Observe that $\G_n(s)$ is obtained by $\HCM_n(\bd^{(w)},\bd^{(b)})$ by keeping all white edges and keeping each black edge independently with probability $p = 1-e^{-s}$. This is the inhomogeneous percolation we refer to in the title.

\subsection{Assumptions and Main Results}

In this work we are interested in the behavior of the graph when both the white and black half-edges have heavy-tailed degree distributions. In order to keep the notation managable we will omit the index $n$ from the degrees, and simply write $\bd^{(w)}$ instead of $\bd^{(w),n}$ for the sequence. We also introduce the following scaling sequences:
\begin{align}\label{eqn:constants}
    \displaystyle a_n& \displaystyle =n^{1/(\tau-1)} L(n) & \displaystyle  b_n &=  \displaystyle \frac{n^{(\tau-2)/(\tau-1)}}{L(n)} &  \displaystyle c_n&= \displaystyle \frac{n^{(\tau-3)/(\tau-1)}}{L(n)^2}
\end{align} where $L$ is a slowly varying function at $\infty$ and $\tau\in(3,4).$

We work under the following assumptions.
\begin{assumption}[Degree Assumptions]\label{ass:Degree} 
The constant $\tau\in(3,4)$, the function $L$ is slowly varying at $\infty$ and the sequences $a_n,b_n,c_n$ are as in~\eqref{eqn:constants}. The white degree sequence $\bd^{(w)}$ and black degree sequence $\bd^{(b)}$ satisfy
\begin{enumerate}[(i)]
\item \label{enum:AssDegree1} (\textit{Arrangement}) The degrees are arranged so that $i\mapsto a_n^{-1} d_i^{(w)}+b_n^{-1} d_i^{(b)}$ is non-increasing.
    \item \label{enum:AssDegree2}(\textit{Hubs}) For each fixed $i\ge 1$, 
    \begin{align*}
        \frac{d_i^{(w)}}{a_n} &\longrightarrow \theta_i>0 & \frac{d^{(b)}_i}{b_n}&\longrightarrow \beta_i\ge 0
    \end{align*}
    where $\btheta = (\theta_1,\theta_2,\dotsm)\in \ell^3\setminus \ell^2$ and $\bbeta = (\beta_1,\beta_2,\dotsm)\in \ell^2$ are such that $\langle \btheta,\bbeta\rangle = \sum_{i=1}^\infty \theta_i\beta_i<\infty$.
    \item \label{enum:AssDegree3} (\textit{Degree Regularity}) There is some integer valued random variable $D$ with $\PR(D\ge 1) = 1$ and $\PR(D=1)>0$ such that for all $k=1,2,\dotsm$, 
    \begin{equation*}
        \frac{\#\{j: d_j^{(w)}= k\}}{n} \longrightarrow \PR(D=k).
    \end{equation*}  Moreover,
    \begin{align*}
        \frac{1}{n}\sum_{i=1}^nd_i^{(w)} &\longrightarrow \kappa:=\E[D], & \frac{1}{n} \sum_{i=1}^n (d_i^{(w)})^2 &\longrightarrow \E[D^2]<\infty, \\\limsup_{K\to\infty} \limsup_{n\to\infty} &\frac{1}{a_n^3} \sum_{i=K+1}^n (d_i^{(w)})^3 = 0 &  \lim_{n\to\infty} \frac{1}{n} \sum_{i=1}^n d_i^{(w)}d_i^{(b)} &= \langle \btheta,\bbeta\rangle + \alpha\E[D]\\
        \limsup_{K\to\infty} \limsup_{n\to\infty}& \frac{1}{b_n^2} \sum_{i=K+1}^n (d_i^{(b)})^2 = 0.
    \end{align*} 
    \item  \label{enum:AssDegree4}(\textit{Criticality}) There is some $\lambda\in \R$ such that
    \begin{equation*}
        \nu_n^{(w)}(\lambda):= \frac{ \sum_{i\in[n]} d^{(w)}_i(d_i^{(w)}-1)}{ \sum_{i\in[n]} d_i^{(w)}} = 1+\lambda c_n^{-1} + o(c_n^{-1}).
    \end{equation*}
    \item \label{enum:AssDegree5}(\textit{Number of Black Edges}) $\ell_n^{(b)} = \sum_{i=1}^{n} d_i^{(b)} = \Theta(n)$ as $n\to\infty$ and $\sum_{i=1}^n (d_i^{(b)})^2 = \Theta(n) = \Theta(\ell_n^{(b)})$.
\end{enumerate}
\end{assumption}

When all the black degrees $d^{(b)}$ are zero (which formally violates Assumption~\ref{ass:Degree}\eqref{enum:AssDegree5}), this model is the configuration model and the scaling regime in Assumption~\ref{ass:Degree}(i-iv) was studied by Dhara et. al. in \cite{DvdHvLS.20}. In this case, the graph $\HCM_n(\bd^{(w)},\boldsymbol{0})$ is just the configuration model $\CM(\bd^{(w)})$. Therein, the authors show that the scaling limits of component sizes are described by the excursion lengths of certain thinned L\'{e}vy process $X$. We will see below in Theorem~\ref{thm:main1}, that the number of incident black half-edges in $\G_n(0)$ are described using an additional \textit{thinned subordinator} $Y$. More precisely, let $\xi_j$ be an independent collection of rate $\theta_j$ exponential random variables. Define
\begin{align}
    X(t) &= \sum_{i=1}^\infty \theta_i \left(1_{[\xi_i\le \kappa t]} - \frac{\theta_i}{\kappa} t \right) + \lambda t, &
    Y(t) &= \sum_{i=1}^\infty \beta_i 1_{[\xi_i\le \kappa t]} + \alpha t\label{eqn:XinfDef}.
\end{align} 
A priori, it need not be the case that $X$ nor $Y$ is well-defined; however, under Assumption~\ref{ass:Degree}\eqref{enum:AssDegree2} both $X$ and $Y$ are. For $X$ it follows from \cite{AL.98} and for $Y$, Fubini-Tonelli implies
\begin{equation*}
    \E[Y(t)] = \alpha t + \sum_{i=1}^\infty \beta_i \PR(\xi_i\le \kappa t) \le \alpha t+ \sum_{i=1}^\infty \beta_i\theta_i \kappa t <\infty
\end{equation*} as $\langle \btheta,\bbeta\rangle<\infty$.

For a c\`adl\`ag function $f$ without negative jumps, we define an excursion interval of $f$ to be a non-empty interval $(l,r)\subset\R_+$ such that
\begin{align*}
f(l-)  &= f(r-) = \inf_{t\le r} f(t) &\textup{and}&& f(t) &> \inf_{s\le l} f(s)\qquad\forall t\in(l,r) .
\end{align*} let $\EE(f)$ denote the collection of excursion intervals of a function $f$. If possible, we define $\LL^\downarrow(f)$ as the ordered sequence of lengths of excursion intervals (i.e. $r-l$) arranged in decreasing order. We now recall a result of Dhara et. al. \cite{DvdHvLS.20} about convergence.
\begin{theorem}[Dhara et. al.{\cite{DvdHvLS.20}}]\label{thm:dvdhls}
Suppose $\bd^{(w)}$ satisfies Assumption~\ref{ass:Degree}. Let $(\cC_n(1),\cC_n(2),\dotms)$ denote the connected components of $\CM(\bd^{(w)})$ listed in decreasing order of their cardinality. Then the
\begin{equation*}
\left(\#\cC_n(j);j\ge 1 \right) \weakarrow \LL^\downarrow (X)\qquad\textup{ in }\ell^2_\downarrow.
\end{equation*}
\end{theorem}


Recall that Theorem \ref{thm:MC} there are two convergent sequences $\bx_n,\by_n\in\ell^2$, while in the above Theorem \ref{thm:dvdhls} there is just a single $\ell^2$ convergence. In order to keep track of the additional weight vector, we will need to introduce a vector-valued analogue that allows us to keep track of the relevant weight. We do this using an idea that is further explored in the forthcoming work \cite{CKL.22}. We write $\ell^{p,q}(\R^2_+)$ for the collection of sequences $\bw = (w_i;i\ge 1) \in (\R^2_+)^\infty$ equipped with the norm
\begin{equation*}
    \|\bw \|_{p,q} =\left(\sum_{i=1}^\infty \|w_i\|_q^{p}\right)^{1/p}= \left(\sum_{i=1}^\infty (|x_i|^q + |y_i|^q)^{p/q}\right)^{1/p}, \qquad  w_i = (x_i,y_i). 
\end{equation*} Obviously $\ell^{p,q}:=\ell^{p,q}(\R^2_+)$ for $p,q\in [1,\infty]$ is closed cone in the separable Banach space of sequences of $\R^2$ vectors with whose $q$-norms are in $\ell^p$. Hence $\ell^{p,q}$ Polish with the subspace topology. Whenever $\bw\in \ell^{p,q}$ we write $\bw = (\bx,\by)$ where $\bx,\by\in \ell^p$ are the obvious projections.

To state our next theorem, if $f:\R_+\to \R$ is such that $\LL^\downarrow(f)$ is well-defined, we denote by $\Delta_i(f) = r_i-l_i$ the length of the $i^\text{th}$ longest excursion interval $(l_i,r_i)$, with ties broken by order of appearance. If $g$ is a non-decreasing function, we let $\Delta_i^g(f) = g(r_i) - g(l_i)$ and we will write $\gamma_i(f,g)\in \R^2_+$ as the element 
\begin{equation}\label{eqn:gamma_iDefinition}
    \gamma_i(f,g) = \left( \Delta_i(f), \Delta_i^{g}(f)\right).
\end{equation} 
We write
\begin{equation*}
    \Gamma^\downarrow (f,g) = \left( \gamma_1(f,g),\gamma_2(f,g),\dotsm\right)\in (\R_+^2)^\infty.
\end{equation*} We note that this may not even take values in $\ell^{\infty,\infty}$ depending on the growth of $g$. 

Our next result describe the size and number of incident black half-edges for the HCM under Assumption~\ref{ass:Degree}. We write $H^{(b)}(A)$ for the collection of black half-edges incident to the vertices in $A$.

\begin{theorem}\label{thm:main1}
Suppose Assumption~\ref{ass:Degree}. 
Let $\cC_n(1),\cC_n(2),\dotsm$ be the connected components of $\G_n(0)$ listed in decreasing order of their cardinality. Then 
\begin{equation*}
    \left(\frac{1}{b_n}(\#\cC_n(j), \#H^{(b)}(\cC_n(j)));j\ge 1 \right) \weakarrow \Gamma^\downarrow (X,Y)\qquad\textup{ in }\ell^{2,2}.
\end{equation*}
\end{theorem}

Theorem~\ref{thm:main1} above is about the convergence of the number of black half-edges before any of them are paired. In particular, it does not tell us much about the component sizes of the graph $\G_n( s)$ for any $s> 0$. Our next theorem tells us that the number of vertices in the graph $\G_n(s)$ when $s = \mu\gamma c_{n}^{-1}$ converges to the MCMW.

\begin{theorem} \label{thm:main2} Assume Assumption~\ref{ass:Degree}. In addition suppose
\begin{equation} \label{eqn:dBconve}
    \frac{1}{n}\sum_{i} d_i^{(b)} \longrightarrow \gamma\in(0,\infty)\qquad\textup{ and  }\qquad \textup{either }\beta_i>0\textup{ i.o.  or  }\alpha>0.
\end{equation}
    For each $\lambda\in \R$ and $\mu\ge 0$ exists an $\ell^2_\downarrow$-valued random variable $\bzeta^{\HCM}(\lambda,\mu) $ such that for the ordered list $(\cC_n(j);j\ge 1)$ of connected components of $\G_n(\mu\gamma c_n^{-1}),$ listed from largest to smallest, satisfy
    \begin{equation*}
        \left(b_n^{-1}\#\cC_n(j);j\ge 1\right) \weakarrow \bzeta^{\HCM}(\lambda,\mu).
    \end{equation*}
    More precisely, 
    \begin{equation*}
        \bzeta^{\HCM}(\lambda, \mu) = \MC_2(\Gamma^\downarrow(X,Y), \mu).
    \end{equation*}
\end{theorem}

\section{Discussion and Overview}

\subsection{Connection with the HCM in \cite{vdHvLS.18}}

Our interpretation of the HCM is slightly different than the interpretation in \cite{vdHvLS.18}. Therein, the authors first consider some communities $(F_j  = (V_j,E_j,\bd_j);j\in[n])$ of various sizes $s_i = \#V_i$. They then pair the vertices in $\sqcup_j V_j$ according to the configuration model. The resulting cardinality of a connected component is then the sum of the corresponding sizes $s_i$. This is precisely what $\#H^{(b)}(\cC)$ counts, after identifying $s_i$ with $d_i^{(b)}$.

In order to compare our Theorem~\ref{thm:main1} with some of results in \cite{vdHvLS.18}, we recast their results into our model. 
\begin{assumption}\label{ass:finiteThird}
    Let $\bd^{(w)}$ and $\bd^{(b)}$ be degree sequences of length $n\to\infty$. Let $D_n^{(w)}$ (resp. $D^{(b)}_n)$ for the white (resp. black) degree of a uniformly chosen vertex of $\HCM_{n}(\bd^{(w)},\bd^{(b)})$. Suppose the following
    \begin{enumerate}[(i)]
        \item There exists random variables $D^{(w)}$ and $D^{(b)}$ such that $(D^{(w)}_n,D^{(b)}_n)\weakarrow (D^{(w)},D^{(b)})$, where $\E[(D^{(w)})^3]<\infty$, $\E[D^{(w)}D^{(b)}]<\infty$ and $\E[D^{(b)}]<\infty$.
        \item The following convergences hold as $n\to\infty$
        \begin{align*}
            \E[D^{(b)}_n] &\longrightarrow \E[D^{(b)}]&
            \E[D_n^{(b)} D^{(w)}_n] &\longrightarrow \E[D^{(b)}D^{(w)}]\\
            \E[ (D_n^{(w)})^3] &\longrightarrow \E[(D^{(w)})^3].
        \end{align*}
        \item $\PR(D^{(w)}=0)<1$ and $\PR(D^{(b)} = 1)\in(0,1)$.
        \item There is some $\lambda\in \R$ such that
        \begin{equation*}
            \nu_{n}^{(w)}(\lambda) = \frac{\sum_{i\in[n]} d^{(w)}_i (d^{(w)}_i-1)}{\sum_{i\in[n]} d_i^{(w)}} = 1+\lambda n^{-1/3}+o(n^{-1/3}).
        \end{equation*}
        \item $\max_{i\in[n]} d_i^{(b)} = o(n^{2/3}/\log(n))$.
    \end{enumerate}
\end{assumption}

The main theorems in \cite{vdHvLS.18} can be summarized as follows. Let $B$ be a standard linear Brownian motion and write
\begin{equation*}
    W^\lambda(t) = \frac{\sqrt{\eta}}{\mu} B(t) + \lambda t - \frac{\eta t^2}{2\mu^3}
\end{equation*}
where $\eta = \E[(D^{(w)})^3]\E[D^{(w)}] - \E[(D^{(w)})^2]$ and $\mu = \E[D^{(w)}]$. Also let $\beta = \frac{\E[D^{(w)} D^{(b)}]}{\E[D^{(w)}]}.$ 
\begin{theorem}[van der Hofstad et. al. {\cite{vdHvLS.18}}]\label{thm:finiteThird}
    Suppose the degree distributions $\bd^{(w)}$ and $\bd^{(b)}$ satisfy Assumption~\ref{ass:finiteThird} and let $(\cC_n(j);j\ge1)$ denote the connected components of $\G_n(0)$ listed in decreasing cardinality. Then, in the product topology,
    \begin{equation*}
        \left(n^{-2/3} \#H^{(b)}(\cC_n(j));j\ge 1\right) \weakarrow \beta \LL^\downarrow (W^\lambda).
    \end{equation*}
    The convergence holds in $\ell^2_\downarrow$ if and only if $\E[(D^{(b)}_n)^2] = o(n^{1/3})$.
\end{theorem}

Let us point out some major differences between Theorems~\ref{thm:main1} and~\ref{thm:finiteThird}. The first is that in Theorem~\ref{thm:finiteThird}, the number of incident black half-edges of large components converges to a constant multiple of the excursion lengths of a Brownian motion with parabolic drift, while the convergence in Theorem~\ref{thm:main1} is to the increments of a process $Y$ which contains jumps over the excursion intervals of a different process $X$. These can be viewed as the same result where the corresponding thinned subordinator in Theorem~\ref{thm:finiteThird} is $U(t) = \beta t$. The next is that in order to obtain convergence in the product topology in Theorem~\ref{thm:finiteThird}, one needs to assume that the $d_i^{(b)}$ are all of order \textit{strictly smaller} than $n^{2/3}$, the size of the connected components in $\G_n(0)$ \cite{DvdHvLS.17}. This is not the case with Theorem~\ref{thm:main1} where the degrees $d_i^{(b)}$ can be of order $\Theta(b_n)$. Roughly speaking, the finite third moment case requires \textit{mesoscopic} communities, but the infinite third moment case can handle \textit{macroscopic} communities.

\subsection{Overview of The Article} 

In Section~\ref{sec:MCMW}, we establish the Feller-type property for $\MC_2(\bx,\by,t)$ found in Theorem~\ref{thm:MC}. This broadly follows the approach that Aldous used to establish the Feller property of $\MC_1(\bx,t)$ in \cite{Aldous.97}. We remark that most of Aldous's work is easily generalized to our more abstract setting, suggesting that maybe there is a more general coalescent process lurking in the background. It would be interesting to see if this can be used to shed light on recent work by Konarovskyi and Limic \cite{KL.21}, as well as their joint work with the author in the forthcoming \cite{CKL.22}. We also establish Proposition~\ref{prop:graphConvergence}, which formalizes some ideas that appear in prior works on critical random graphs and the multiplicative coalescent, see for example \cite{DvdHvLS.20,BDvdHS.20,DvdHvLS.17}. Roughly speaking, Proposition~\ref{prop:graphConvergence} states that if you can couple a graph $G_n$ with a graph $G_n'$ such that $G_n\subset G'_n$ and the edge probabilities are asymptotically the same, then if the component masses of $G_n'$ converge to some limit $\bX$, then so do the component masses of $G_n$. This makes establishing Theorem~\ref{thm:main2} much easier. 

In Section~\ref{sec:Exploration}, we describe the exploration process for the HCM which is essentially the same as \cite{DvdHvLS.20}. This encodes the number of vertices in the connected components of $\G_n(0)$, along with the number of incident black half-edges using relatively simple stochastic processes $(X_n,Y_n)$. We show in Proposition~\ref{prop:ConvergenceWalks} that $(X,Y)$ defined in~\eqref{eqn:XinfDef} appear as the natural scaling limits for these walks. Unfortunately, this Proposition~\ref{prop:ConvergenceWalks} does not follow directly from \cite{DvdHvLS.20} so we need to include a large part of the analysis. However, since the argument is similar to that found in \cite{DvdHvLS.20}, some parts are only sketched or even omitted. 

Afterwards, we need to extend Proposition~\ref{prop:ConvergenceWalks} to the convergence in Theorem~\ref{thm:main1}. This requires us to do two things. The first is that we show that all the macroscopic components and components with a large number of black half-edges are discovered early in our exploration. This is our Lemma~\ref{lem:LateComponentsAllSmall}. To establish this, we need to show that Assumption~\ref{ass:Degree} can control a susceptibility-type functional on the graph. The second is that a particular point process constructed using the graph exploration converges to a limiting point process described using $(X,Y)$ and this convergence extends to $\ell^2$ convergence in Theorem~\ref{thm:main2}. We do this by extending some results of Aldous in \cite{Aldous.97}.

Finally, in Section~\ref{sec:Inhomo} we establish Theorem~\ref{thm:main2} by showing that $\G_n(s)$ is sufficiently close to the MCMW. 

\section{The Multiplicative Coalescent with Mass and Weight}\label{sec:MCMW}

\subsection{Basic Lemmas and Notations}

In this section, we recall some important lemmas established by Aldous in \cite{Aldous.97}. We also re-establish a simple compactness result, also used in \cite{CKL.22}, which is a simple extension of a result by Aldous and Limic in \cite{AL.98}

We write $\bx\le \bx'$ if for all $l$ the $l^\text{th}$ coordinate of $\bx$ is at most the $l^\text{th}$ coordinate of $\bx'$, that is, $x_l\le x_l'$. We write $\bx\preceq \bx'$ if the entries of $\bx$ as a multiset $\{\{x_1,x_2,\dotsm\}\}$ can be rewritten as $\{\{y_{i,j}: i,j\ge 1\}\}$ so that $\sum_{j} y_{i,j}\le x'_i$. It is elementary to see that if $\bx\preceq \bx'$ then $\|\bx\|^2\le \|\bx'\|^2$ using $(x+y)^2 \ge x^2+y^2$. The next lemma (Lemma 17 in \cite{Aldous.97}) gives us a particular example where $\preceq$ comparison appears. Technically, the $\preceq$ notation was introduced after \cite{Aldous.97} and so Lemma 17 does not mention it. This new observation is a trivial.

\begin{lemma}\label{lem:Lemma17Aldous}
    Suppose that $G$, $G'$ are graphs on the same vertex set, and that each graph has vertex masses $(x_i)$, $(x'_i)$ respectively. Suppose that $G\subset G'$. Let $\mathbf{a}$ (resp.$ \mathbf{a}'$) be the masses of the connected components of $G$ (resp. $G'$) listed in decreasing order. Then $\mathbf{a}\preceq \mathbf{a}'$ and
    \begin{equation*}
        \|\mathbf{a}'-\mathbf{a}\|^2 \le \|\mathbf{a}'\|^2 - \|\mathbf{a}\|^2,
    \end{equation*} provided $\|\mathbf{a}\|^2<\infty$.
\end{lemma}

We need the following compactness result. This is contained in \cite{CKL.22}, and we include a proof here.

\begin{lemma}\label{lem:compact2}
Let $K\subset\ell^2_\downarrow$ be a compact subset. Then
\begin{equation*}
    A = \{\bx: \exists\, \by\in K \text{ s.t. }\bx\preceq \by\}\quad\text{ is pre-compact in  }\ell^2_\downarrow.
\end{equation*}
\end{lemma}

\begin{proof}
We follow the proof of Proposition~36~(ii) in \cite{AL.98}.
Since $\bx\preceq\by$ implies $\|\bx\|\le \|\by\|$, the compactness of $K$ implies that $A$ is bounded in $\ell^2_\downarrow$. Thus, to show that it is pre-compact, so it suffices to show the following:
\begin{equation*}
    \lim_{k\to\infty} \sup_{\bx\in A}\sum_{i\ge k}x_i^2 =0.
\end{equation*}
Fix $\eps>0$. Note that if $y_{i,j}\ge 0$ and $\sum_{j} y_{i,j} \le y_i$ then both 
\begin{equation*}
    \sum_{j} y_{i,j}^2\le y_i^2\qquad\textup{and}\qquad \sum_{j} y_{i,j}^2 1_{[y_{i,j}\le \eps]}\le \eps y_i.
\end{equation*} Hence for any fixed $k$
\begin{align*}
    \sup_{\bx\in A} \left\{\sum_{i=1}^\infty x_{i}^2 1_{[x_i\le \eps]}\right\} \le \sup_{\by\in K}\left\{ \eps \sum_{i< k} y_i + \sum_{i\ge k} y_i^2\right\}.
\end{align*}
Taking the limit superior as $\eps\downarrow 0$ gives
\begin{align*}
    \limsup_{\eps\downarrow0}& \sup_{\bx\in A} \left\{\sum_{i=1}^\infty x_{i}^2 1_{[x_i\le \eps]}\right\} \le \limsup_{\eps\downarrow 0} \left(\sup_{\by\in K}\left\{  \eps \sum_{i< k} y_i\right\} + \sup_{\by\in K}\left\{ \sum_{i\ge k} y_i^2\right\} \right)\\
    &= \sup_{\by\in K}  \left\{ \sum_{i\ge k} y_i^2 \right\}.
\end{align*}
Taking $k\to\infty$ and use compactness of $K$, we get $\sup_{\by\in K} \{\sum_{i\ge k} y_i^2\}\longrightarrow 0$ by, for example, Dini's theorem. Hence
\begin{equation}
    \label{eqn:limsupCompact1}
    \limsup_{\eps\downarrow 0} \,\sup_{\bx\in A} \left\{\sum_{i=1}^\infty x_i^2 1_{[x_i\le \eps]} \right\} = 0.
\end{equation}

Also note that $x_1\ge x_2\ge \dotsm\ge 0$ for all $\bx\in A$ and $\sup_{\by\in K}\|\by\|^2_2<\infty$. Now for each fixed $\eps>0$, for all $\bx\in A$ we must have that $x_i\le \eps$ for all $i\ge \frac{1}{\eps^2} \sup_{\by\in K}\|\by\|^2_2.$
Hence, for all $\eps>0$,
\begin{equation*}
    \limsup_{k\to\infty} \sup_{\bx\in A} \sum_{i\ge k} x_i^2 \le \sup_{\bx\in A}\sum_{i=1}^\infty x_i^21_{[x_i\le \eps]}.
\end{equation*} Taking the limit as $\eps\downarrow 0$ gives the desired claim using~\eqref{eqn:limsupCompact1}.
\end{proof}

The following corollary is immediate.
\begin{corollary}\label{cor:compactCKL}
    Suppose $\bX_n,\bZ_n$ are random elements in $\ell^2_\downarrow$ are coupled so that $\bX_n\preceq \bZ_n$ for all $n$. If $(\bZ_n;n\ge 1)$ is tight, then $(\bX_n;n\ge 1)$ is also tight.
\end{corollary}

Before continuing to couplings, we recall a useful lemma by Aldous \cite[Lemma 20]{Aldous.97}, which we recall without proof. Recall that we write $\MC_1(\bx,t)$ as the mulitplicative coalescent at time $t$.
\begin{lemma}\label{lem:Lemma20Aldous}
Let $\bx\in \ell^2$ and let $s> \|\bx\|^2.$ Then
\begin{equation*}
    \PR\left(\|\MC_1(\bx,t)\|^2 > s\right) \le \frac{ts \|\bx\|^2}{s-\|\bx\|^2}.
\end{equation*}
\end{lemma}

\subsection{Graphical Couplings}
We now recall two couplings of Aldous \cite{Aldous.97}. Consider the complete graph $K_\infty$ on vertices index by $\N = \{1,2,\dotsm\}$. To each edge $e = \{i,j\}\in E(K_\infty)$ generate an independent rate $1$ exponential random variable $\xi_e = \xi_{ij} = \xi_{ji}$. For any $\bx,\by\in \ell^2_+$, write $\sW^G(\bx,\by,t)$ for the subgraph obtained by including the edge $\{i,j\}$ if and only if $\xi_{ij}\le y_iy_jt$. For a connected component $A\subset \sW^G(\bx,\by,t)$ we define its mass as $\sM(A) = \sum_{i\in A} x_i$. 

Clearly we have $\MC_2(\bx,\by,t) = (\sM_i(\bx,\by,t);i\ge1)$ the vector of component masses of $\sW^G(\bx,\by,t)$ listed in decreasing order. For any collection of vectors $\bx,\by,\bx',\by'$ we can couple $\sW^G(\bx,\by,t)$ and $\sW^G(\bx',\by',t)$ so that they live in the same probability space. We call this the \textit{$\xi$-coupling.}

Obviously the $\xi$-coupling of $\sW^G(\bx,\by,t)$ and $\sW^G(\bx+\by,\bx+\by,t)$ implies that the connected components of $\sW^G(\bx+\by,\bx+\by,t)$ are unions of connected components of $\sW^G(\bx,\by,t)$ and that each vertex has larger mass in $\sW^G(\bx+\by,\bx+\by,t)$ than in $\sW^G(\bx,\by,t)$. In particular, 
\begin{equation*}
    \MC_2(\bx,\by,t)\preceq \MC_2(\bx+\by,\bx+\by,t) = \MC_1(\bx+\by,t).
\end{equation*} The above coupling together with the fact that the standard multiplicative coalescent lives in $\ell^2_\downarrow$ \cite{Aldous.97} implies the following.
\begin{claim}
    If $\bx,\by\in \ell^2$. Then $\MC_2(\bx,\by,t)\in \ell^2_\downarrow$ for all $t$. 
\end{claim}

The next lemma follows directly from Lemma 17 (recalled above in Lemma~\ref{lem:Lemma17Aldous}) and Corollary 18 in \cite{Aldous.97}. We state it here without proof.
\begin{lemma}\label{lem:cor18Aldous}
    Suppose that $\bx\le \bx'$ and $\by\le \by'.$ Then \begin{enumerate}
        \item $\|\MC_2(\bx',\by',t) - \MC_2(\bx,\by,t)\|^2 \le \|\MC_2(\bx',\by',t)\|^2-\|\MC_2(\bx,\by,t)\|^2$, when $\MC_2(\bx,\by,t)\in\ell^2$;
        \item If $t_1<t_2$ then $\|\MC_2(\bx,\by,t_2) - \MC_2(\bx,\by,t_1)\|^2\le \|\MC_2(\bx,\by,t_2)\|^2-\|\MC_2(\bx,\by,t_1)\|^2$ when $\MC_2(\bx,\by,t_1)\in\ell^2$;
        \item If $\bx^{(n)} = (x_1,x_2,\dotms,x_k,0,\dotms)$, $\by^{(n)} = (y_1,y_2,\dotsm,y_k,0,\dotsm)$ and $\MC_2(\bx,\by,t)\in\ell^2$ then
        \begin{equation*}
            \|\MC_2(\bx^{(n)},\by^{(n)},t)\|^2\uparrow  \|\MC_2(\bx,\by,t)\|^2\qquad\text{and} \qquad \MC_2(\bx^{(n)},\by^{(n)},t)\longrightarrow \MC_2(\bx,\by,t).
        \end{equation*}
    \end{enumerate}
\end{lemma}

We will need the following lemma as well. 
\begin{lemma}\label{lem:helpLem1}
If $\bx,\by\in \ell^2$ and $t>0$ then for the $\xi$-coupling
    \begin{align*}
    \|\MC_2&(\bx,\by,t)\|^2 \le \|\MC_2(\bx+\by,\by,t)\|^2 - \|\by\|^2-2\langle\bx,\by\rangle\\
    &\le \|\MC_2(\bx+\by,\bx+\by,t)\|^2-  \|\by\|^2-2\langle\bx,\by\rangle
    \end{align*}
\end{lemma}
\begin{proof}
The last inequality follows from Lemma~\ref{lem:Lemma17Aldous} and we only show the first. 

If $A\subset\N$ then
\begin{align*}
\left(\sum_{i\in A} x_i+y_i\right)^2 &=    \left(\sum_{i\in A} x_i\right)^2+  \left(\sum_{i\in A} y_i\right)^2 + 2\left(\sum_{i\in A} x_i\right)\left( \sum_{i\in A} y_i\right)\\
&\ge \left(\sum_{i\in A} x_i\right)^2 + \sum_{i\in A} y_i^2 + 2\sum_{i\in A} x_i y_i.
\end{align*}
Suppose that $\boldsymbol{\pi} = (\pi_1,\pi_2,\dotms)$ is a partition of $\N$. Then, the above analysis shows
\begin{equation}\label{eqn:citeLater}
    \sum_{p=1}^\infty \left( \sum_{i\in \pi_p} x_i\right)^2 \le \sum_{p=1}^\infty \left( \sum_{i\in \pi_p} x_i+y_i\right)^2 -\sum_i y_i^2 - 2\sum_i x_iy_i,
\end{equation} provided that the latter two infinite summations are finite, which holds whenever $\bx,\by \in \ell^2$. In particular, this holds a.s. partition of $\N$ given by the connected components of $\sW^G(\bx,\by,t)$. 
\end{proof}

There is another coupling that Aldous introduces, which will be useful later. It is called the \textit{subgraph coupling}. Here we have $\bx = (x_\alpha; \alpha\in A)$ and $\by = (y_\alpha;\alpha\in A)$ for a countable index set $A$ and a subcollection $\bx' = (x_\alpha; \alpha \in A')$ and $\by' = (y_\alpha; \alpha\in A')$ for some subset $A'\subset A$. In this setting, we can also couple $\sW^G(\bx,\by,t)$ and $\sW^G(\bx',\by',t)$ using the same exponential random variables, and $\sW^{G}(\bx',\by',t)$ is the induced subgraph on the vertices $A'$.   

\subsection{Sufficient Condition for Theorem~\ref{thm:MC}}

A key step in the proof that the multiplicative coalescent is Feller is Lemma 22 in \cite{Aldous.97}. That lemma gives a relatively easy to verify condition that is sufficient to conclude the Feller property. The result below is analogous and its proof is almost identical to Aldous' proof of Lemma 22 in \cite{Aldous.97}. It only requiring minor changes involving the weight and we include it. As writing out the moments $\|\bx\|^2$ will become tedious quickly, we write $\bS(\bx,\by,t) = \|\MC_2(\bx,\by,t)\|^2.$

\begin{lemma}\label{lem:Lemm22Aldous}
    Suppose that $\bx^{(n)}\to \bx$ and $\by^{(n)}\to \by$ in $\ell^2_+$. Then for the above $\xi$-coupling we have 
    \begin{equation*} \liminf_{n\to\infty} \bS(\bx^{(n)},\by^{(n)},t)\ge \bS(\bx,\by,t).
    \end{equation*}
    Moreover if one can couple $\MC_2(\bx^{(n)},\by^{(n)},t)$ and $\MC_2(\bx,\by,t)$ in such a way that
    \begin{equation}\label{eqn:coupledSusceptibiliy}
       \lim_{n\to\infty} \PR\left( \bS(\bx^{(n)},\by^{(n)},t) - \bS(\bx,\by,t) >\eps \right) = 0
    \end{equation}
    then $\MC_2(\bx^{(n)},\by^{(n)},t) \weakarrow \MC_2(\bx,\by,t)$.
\end{lemma}
\begin{proof}
We follow the proof of Lemma 22 in \cite{Aldous.97} as well. 

Let $A_{ij}^{(n)}$ (resp. $A_{ij}$) denote the indicator that vertices $i$ and $j$ are in the same component of $\sW^G(\bx^{(n)},\by^{(n)},t)$ (resp. $\sW^G(\bx,\by,t)$). If $i$ and $j$ are in the same component of $\sW^G(\bx,\by,t)$ then there is a finite length path $i = i_0\sim\dotsm\sim i_k  =j$ such that $\xi_{i_l i_{l-1}} \ge ty_{i_l}y_{i_{l-1}} = \lim_{n} t y_{i_l}^{(n)}y_{i_{l-1}}^{(n)}$ for all $l\in[k]$. Therefore
\begin{equation}\label{eqn:AijConv}
    \liminf_{n\to\infty} A_{ij}^{(n)}  \ge  A_{ij}\qquad\textup{ a.s.}
\end{equation}

Let $B^{(n)}_{ij}$ denote the indicator that $i,j$ are in the same component of $\sW^G(\bx^{(n)},\by^{(n)},t)$ and $\sW^G(\bx,\by,t)$. So that $B_{ij}^{(n)} = A_{ij}^{(n)} A_{ij}$ and
\begin{equation}\label{eqn:BijConv}
    \lim_{n\to\infty} B_{ij}^{(n)} = A_{ij}.
\end{equation} Following \cite{Aldous.97}, we say that $i,j$ are in the same \textit{modified} component if $B_{ij}^{(n)} = 1$. We will denote modified components of $\sW^G(\bx^{(n)},\by^{(n)},t)$ by $\cC'$ and will will denote connected components by $\cC$. We will denote the collection of modified components in $\sW^G(\bx^{(n)},\by^{(n)},t)$ by $\sC_n'$ and the collection of connected components by $\sC_n$. We drop the subscript when referring to $\sW(\bx,\by,t)$.

Now by~\eqref{eqn:BijConv} and $\bx^{(n)}\to\bx, \by^{(n)}\to\by$ we have
\begin{align*}
    \lim_{n\to\infty} \sum_{\cC'\in \sC'_n} \left(\sum_{i\in \cC': i\le k} x_i\right)^2 = \sum_{\cC\in \sC}\left( \sum_{i\in \cC: i\le k}x_i \right)^2,
\end{align*}
and
\begin{equation*}
    \lim_{n\to\infty} \sum_{\cC'\in \sC'_n} \left(\sum_{i\in \cC': i\le k} x_i^{(n)}\right)^2 = \sum_{\cC\in \sC}\left( \sum_{i\in \cC: i\le k}x_i \right)^2.
\end{equation*}
By taking $k\to\infty$ we get
\begin{equation*}
    \liminf_{n\to\infty}  \sum_{\cC'\in \sC'_n} \left(\sum_{i\in \cC'} x_i^{(n)}\right)^2 \ge \sum_{\cC\in \sC}\left( \sum_{i\in \cC}x_i \right)^2  =\bS(\bx,\by,t).
\end{equation*}
Since each component of $\sW^G(\bx^{(n)},\by^{(n)},t)$ is the union of modified components, it follows that
\begin{equation*}
    \bS(\bx^{(n)},\by^{(n)},t) \ge \sum_{\cC'\in \sC'_n} \left(\sum_{i\in \cC': i\le k} x_i^{(n)}\right)^2.
\end{equation*} This establishes the first claim.

Now suppose that~\eqref{eqn:coupledSusceptibiliy} holds for some coupling. If it holds for some coupling then it must hold for the $\xi$-coupling above and so
\begin{equation*}
    \bS(\bx^{(n)},\by^{(n)},t)\longrightarrow \bS(\bx,\by,t)\qquad \textup{ in probability.}
\end{equation*}By passing to a subsequence (still denoted by $n$) we can suppose that $\bS(\bx^{(n)},\by^{(n)},t)\longrightarrow \bS(\bx,\by,t)$ a.s. 

To see that $\MC(\bx^{(n)},\by^{(n)},t)\to \MC(\bx,\by,t)$ a.s. for this subsequence, it suffices to order the component masses of $\sW^G(\bx^{(n)},\by^{(n)},t)$, say $\bZ^{(n)} = (Z_1^{(n)},Z_2^{(n)},\dotsm)$, and $\sW^G(\bx,\by,t)$, say $\bZ = (Z_1,Z_2,\dotsm)$, in some way so that \cite[equation (60)]{Aldous.97}
\begin{align}\label{eqn:AldousEQN60}
    \sum_{j} (Z_j^{(n)})^2 &\to \sum_j Z_j^2&\text{and}&& \liminf_{n\to\infty} Z_j^{(n)}&\ge Z_j\qquad\textup{a.s.}
\end{align} The former holds as $\bS(\bx^{(n)},\by^{(n)},t) \to \bS(\bx,\by,t)$ a.s. We claim that latter holds for the specific choice
\begin{equation*}
    Z_j^{(n)} = \begin{cases}
        0&:\exists i<j \textup{ s.t. }A_{ij}^{(n)} = 1\\
        x_j+\sum_{i} x_i^{(n)}A_{ij}^{(n)}&:\textup{otherwise},
    \end{cases}
\end{equation*} and $\bZ$ defined similarly. Now suppose that $j$ is a vertex in $\sW^G(\bx,\by,t)$. If $j$ is the smallest labeled vertex, then Fatou's lemma, ~\eqref{eqn:AijConv} and the assumption that $\bx^{(n)}\to\bx$ imply that almost surely on $\{A_{ij} = 0, \forall i<j\}$
\begin{align*}
    Z_j &= x_j + \sum_{i >j} x_i A_{ij} \le \liminf_{n\to\infty} x_j^{(n)} + \sum_{i>j} \liminf_{n\to\infty} x_j^{(n)} A_{ij}^{(n)}\\
    &\le \liminf_{n\to\infty} \left(x_j^{(n)} + \sum_{i>j} x_j^{(n)}A_{ij}^{(n)}\right) = \liminf_{n\to\infty} Z_j^{(n)}.
\end{align*} If $j$ is not the smallest labeled vertex, then the claim is trivial as $Z_{j}^{(n)}\ge 0$ and $Z_j = 0$.
\end{proof}

\subsection{Proof of Theorem~\ref{thm:MC}}

We now turn to the proof of Theorem~\ref{thm:MC}. In order to do we wish exhibit some coupling so that~\eqref{eqn:coupledSusceptibiliy} in Lemma~\ref{lem:Lemm22Aldous} holds. Again, most of the proof proceeds as the proof in \cite{Aldous.97} and a key step in establishing the coupling is Lemma 23 therein. This is the first step where we do not have essentially the same conclusion as the analogous result in \cite{Aldous.97}. In this case, our bound is slightly worse. We recall that the $\bx$ terms are referred to as mass term and the $\by$ terms are referred to as weight term.
\begin{lemma}\label{lem:Lemma23Aldous}
    Suppose that $\bx$ and $\by$ are of finite length $n<\infty$. Let $m\in[n-1]$. Let $\sB$ denote the random bipartite graph on vertices $[m]\cup [n]\setminus[m]$ where vertices are included independently with probabilities
    \begin{equation*}
        \PR(i\sim j\textup{ in }\sB) = \begin{cases}
            1-\exp(-ty_iy_j) &: 1\le i \le m < j\le n\\
            0&:\textup{ else}.
        \end{cases}
    \end{equation*} Write $\alpha_1 = \sum_{i=1}^m x_i^2$ and $\alpha_2 = \sum_{i=1}^m y_i^2$. Let $(Z_i)_i$ denote the masses of the connected components of $\sB$.
    Then
    \begin{equation}\label{eqn:epsPbound}
        \eps \PR\left(\sum_{i} Z_i^2 > \alpha_1 + \eps\right) \le \left(t (\alpha_1+2\alpha_2 +3\eps) + t^2 (\alpha_1+\alpha_2 + 2\eps )^2 \right)\sum_{k=m+1}^n (x_k+y_k)^2.
    \end{equation}
\end{lemma}

\begin{proof}
Let $\sB_k$ denote the induced subgraph of $\sB$ on the vertex set $[k]$.  Let $Q_m = \sum_{i=1}^m x_i^2$ and $\tilde{Q}_m = \sum_{i=1}^m y_i^2$. Note that since $\sB_m$ is the empty graph on the vertex set $[m]$, it holds that $Q_m$ is the sum of the squares of the component masses of the components of $\sB_m$ and $\tilde{Q}_m$ is the sum of squared weights. More generally, let $Q_k$ (resp. $\tilde{Q}_k$) denote the sum of the squares of component masses (resp. weights) of $\sB_k$. 

Let $A_i$ be the event that $i\sim m+1$ in $\sB$ and hence $\sB_{m+1}$. Note that
\begin{align}
    \label{eqn:QQtilde} Q_{m+1} - Q_{m} &= 2\sum_{i=1}^m x_i x_{m+1} 1_{A_i} + \operatornamewithlimits{{\sum}^*}_{i,j\le m} x_ix_j 1_{A_i\cap A_j}\ge 0\\
   \nonumber \tilde{Q}_{m+1} - \tilde{Q}_{m} &= 2\sum_{i=1}^m y_i y_{m+1} 1_{A_i} + \operatornamewithlimits{{\sum}^*}_{i,j\le m} y_iy_j 1_{A_i\cap A_j}\ge 0
\end{align} where we write that $\Sigma^*$ for the summation over distinct indices. In particular, we see that
\begin{align*}
    \E&\left[Q_{m+1}-Q_m\right] = 2\sum_{i=1}^m x_ix_{m+1} (1-e^{-ty_iy_{m+1}}) + \operatornamewithlimits{{\sum}^*}_{i,j\le m} x_{i}x_j (1-e^{-ty_iy_{m+1}})(1-e^{-ty_jy_{m+1}})\\
    &\le 2t \sum_{i=1}^m x_{i}y_i x_{m+1}y_{m+1} + t^2\operatornamewithlimits{{\sum}^*}_{i,j\le m} x_{i}y_i x_{j}y_j y_{m+1}^2 \\
    &=2t \left(\sum_{i=1}^m x_{i}y_i\right) x_{m+1}y_{m+1} + t^2\left(\operatornamewithlimits{{\sum}^*}_{i,j\le m} x_{i}y_i x_{j}y_j\right) y_{m+1}^2\\
    &\le 2t \sqrt{Q_m \tilde{Q}_m} x_{m+1}y_{m+1}  + t^2 Q_m \tilde{Q}_m y_{m+1}^2\\
    &\le \left(t \sqrt{Q_m \tilde{Q}_m} + t^2 Q_{m}\tilde{Q}_m\right) (x_{m+1}+y_{m+1})^2,
\end{align*}
where we used Cauchy-Schwarz to go from the third to the fourth line. In the last line, we used $(a+b)^2 \ge b^2, 2ab$ for $a,b\ge 0$.
By replacing $x_i$ with $y_i$ we see 
\begin{equation*}
    \E\left[\tilde{Q}_{m+1} - \tilde{Q}_m \right] \le \left(t \tilde{Q}_m + t^2 \tilde{Q}_m^2 \right) (x_{m+1}+y_{m+1})^2.
\end{equation*}

Similar analysis implies
\begin{align*}
    \E&\left[Q_{k+1}-Q_k | \sB_k\right] \le \left(t \sqrt{Q_k \tilde{Q}_k} + t^2 Q_{k}\tilde{Q}_k\right) (x_{k+1}+y_{k+1})^2\\
    \E&\left[\tilde{Q}_{m+1} - \tilde{Q}_m |\sB_k\right] \le \left(t \tilde{Q}_k + t^2 \tilde{Q}_k^2 \right) (x_{k+1}+y_{k+1})^2.
\end{align*}

In particular, the process
\begin{align*}
    M_k = Q_k + &\tilde{Q}_k - Q_m-\tilde{Q}_m  \\
    &- \sum_{j = m}^{k-1}  \left(t \sqrt{Q_j \tilde{Q}_j} + t\tilde{Q}_j + t^2 Q_{j}\tilde{Q}_j+ t^2\tilde{Q}^2\right) (x_{j+1}+y_{j+1})^2\qquad m\le k \le n
\end{align*} is a supermartingale with respect to the filtration $\F_k = \sigma(\sB_k)$ with $\E[M_m] = 0$.

Let $T= \min\{k\ge m: Q_k > Q_m+\eps \textup{  or  } \tilde{Q}_k>\tilde{Q}_m+\eps\}$. The optional stopping theorem implies
\begin{equation*}
    \E\left[M_{T\wedge n}\right] \le \E[M_m] =  0.
\end{equation*} Expanding the left-hand side and using the definition of $T$ 
\begin{align*}
    \E&\left[Q_{T\wedge n}+\tilde{Q}_{T\wedge n} - Q_m-\tilde{Q}_m\right] \\
    &\le \E\left[\sum_{j=m}^{(T\wedge n )- 1}  \left(t \sqrt{Q_j \tilde{Q}_j}+ t\tilde{Q}_j + t^2 Q_{j}\tilde{Q}_j+t^2\tilde{Q}_{j}^2\right) (x_{j+1}+y_{j+1})^2\right]\\
    &\le\left(t \sqrt{(Q_m+\eps)(\tilde{Q}_{m}+\eps)}+ t(\tilde{Q}_m+\eps) + t^2 (Q_m+\eps) (\tilde{Q}_m+\eps)+ t^2 (Q_m+\eps)^2  \right)\\
    &\qquad\qquad\times\sum_{j=m+1}^n (x_j+y_j)^2.
\end{align*}
For any positive numbers $a,b$ it holds that $\sqrt{ab}\le a\vee b \le (a+b)$, $\sqrt{ab}+b\le (a+2b)$, and $ab+b^2 \le (a+b)^2$ and so  
\begin{align*}
   \E&\left[Q_{T\wedge n}+\tilde{Q}_{T\wedge n} - Q_m-\tilde{Q}_m\right]\\
   &\le\left(t(2\tilde{Q}_m +Q_m+3\eps) + t^2 (Q_m+\tilde{Q}_m+2\eps)^2  \right)\sum_{j=m+1}^n (x_j+y_j)^2.
\end{align*} Since $Q_{k}$ and $\tilde{Q}_k$ increase with $k$ by~\eqref{eqn:QQtilde}, $Q_{T\wedge n}+\tilde{Q}_{T\wedge n} - Q_{m} - \tilde{Q}_m> \eps$ on $\{Q_n - Q_m>\eps\}$. Markov's inequality implies~\eqref{eqn:epsPbound}.
\end{proof}

At several points in \cite{Aldous.97}, Aldous extends a result for finite graphs to a result for infinite graphs which he calls ``extension by truncation." Since we will only need this argument for an extension of Lemma~\ref{lem:Lemma23Aldous}, we state this as a separate lemma. The proof is an elementary limiting argument, which we omit.
\begin{lemma}\label{lem:Lemma23Infinite}
Suppose that $A$ and $B$ are two countable index sets and $\bx^A,\by^A\in \ell^2(A\to\R_+)$ and $\bx^B,\by^B\in \ell^2(B\to \R_+)$ are square summable sequences indexed by their superscript. Let $\sB$ denote the random bipartite graph on vertices $A\cup B$ where vertices are included independently with probabilities
    \begin{equation*}
        \PR(i\sim j\textup{ in }\sB) = \begin{cases}
            1-\exp(-ty_\alpha^Ay^B_\beta) &: \alpha\in A, \beta\in B\\
            0&:\textup{ else}.
        \end{cases}
    \end{equation*} Let $(Z_i)_i$ denote the masses of the connected components of $\sB$.
    Then
    \begin{align*}
        \eps \PR&\left(\sum_{i} Z_i^2 > \|\bx^A\|^2 + \eps\right) \\
        &\le \bigg(t (\|\bx^A\|^2+2\|\by^A\|^2 +3\eps) + t^2 (\|\bx^A\|^2+\|\by^A\|^2 + 2\eps )^2 \bigg)\|\bx^B+\by^B\|^2.
    \end{align*}
\end{lemma}

We now turn to the proof of Theorem~\ref{thm:MC} following the approach of Aldous \cite[pgs 842-843]{Aldous.97}
\begin{proof}[Proof of Theorem~\ref{thm:MC}] As already observed, $\bS(\bx,\by,t)<\infty$ a.s. by a coupling with the multiplicative coalescent.

Denote by $\bx^{(n,k)}$ (resp. $\by^{(n,k)}$) a rearrangement of the multiset $\{\{x_i:i\ge 1 \}\}\cup \{\{ x_i^{(n)}: i\ge k\}\}$ (resp. $\{\{y_i:i\ge 1 \}\}\cup \{\{ y_i^{(n)}: i\ge k\}\}$). Here we use the same ``rearrangement map'' $\{i:i\ge 1\}\sqcup\{j: j\ge k\}\to \{i;i\ge 1\}$ for the indices of $\bx^{(n,k)}$ and $\by^{(n,k)}$.  Using the subgraph coupling we can coupling $\sW^{G}(\bx,\by,t)$ and $\sW^{G}(\bx^{(n,k)}, \by^{(n,k)},t)$ and hence $\bS(\bx,\by,t)$ and $\bS(\bx^{(n,k)},\by^{(n,k)},t)$ as well. We claim that
\begin{align}\label{eqn:Aldous61}
    \lim_{k\to\infty} &\limsup_{n\to\infty} \PR\left(\bS(\bx^{(n,k)},\by^{(n,k)},t) - \bS(\bx,\by,t)>\eps\right)=0.
\end{align}
Indeed, for subgraph coupling above, we can set $\F^{(n,k)}$ as the $\sigma$-algebra generated by the induced subgraphs $A,B$ of $\sW^G(\bx^{(n,k)}, \by^{(n,k)},t)$ where the vertices of $A$ are those whose mass corresponds to an $x_i$ term and the vertices of $B$ are those whose mass corresponds to an $x_{i}^{(n)}$ term. Now, we write $\bS'(A)$ for the random variable obtained in this coupling, which is the squared sum of the \textit{masses and weights} of the induced sugbraph $A$. That is
\begin{equation*}
   \bS(\bx,\by,t) = \sum_{\cC\subset A}\left( \sum_{i\in \cC} x_i \right)^2 \qquad\bS'(A) = \sum_{\cC\subset A}\left( \sum_{i\in \cC} x_i+y_i \right)^2,
\end{equation*} where the outside summations are over the connected components $\cC$ of the induced subgraph $\sW^{G}(\bx,\by,t)$ of $\sW^G(\bx^{(n,k)}, \by^{(n,k)},t)$. Similarly define $\bS'(B)$. Lemma~\ref{lem:Lemma23Infinite} then implies
\begin{align*}
   \eps \PR&\left( \bS(\bx^{(n,k)},\by^{(n,k)},t)>\bS(\bx,\by,t)+\eps\big|\F^{(n,k)} \right)\\
   &\le \bigg(t (3\bS'(A) + 3\eps) + t^2 (2\bS'(A)+2\eps)^2\bigg) \bS'(B).
\end{align*} 
We now claim that $\lim_{k\to\infty} \limsup_{n\to\infty} \bS'(B) = 0$ where the limits are in probability. Indeed, by the convergence in $\ell^2$ of $\bx^{(n)}\to \bx$ and $\by^{(n)}\to\by$ we have
\begin{equation*}
    \lim_{k\to \infty} \sup_{n\ge 1} \sum_{i\ge k}(x_i^{(n)} + y_i^{(n)})^2 =0.
\end{equation*} Using the $\xi$-coupling we have for all $n\ge1$ but $k$ sufficiently large
\begin{align*}
    \PR\left(\bS'(B) > s\right) &\le \PR\left(\left\|\MC_1\left(\{\{x_i^{(n)}+y_i^{(n)};i\ge k\}\}, t\right)\right\|^2> s \right) \le \frac{ts \sum_{i\ge k} (x^{(n)}_i+y^{(n)}_i)^2}{s-\sum_{i\ge k} (x^{(n)}_i+y^{(n)}_i)^2},
\end{align*}
where the second line follows from Lemma~\ref{lem:Lemma20Aldous} recalled above. These arguments imply~\eqref{eqn:Aldous61}.

Let us now define $\bx^{[n,k]}=(x_1,\dotsm,x_{k-1}, x_k^{(n)},x_{k+1}^{(n)},\dotsm)$ and similarly define $\by^{[n,k]}$. Using the subgraph coupling and Lemma~\ref{lem:Lemma17Aldous} it is easy to see that $\bS(\bx^{[n,k]}, \by^{[n,k]},t) \le \bS(\bx^{(n,k)}, \by^{(n,k)},t)$ a.s. Hence~\eqref{eqn:Aldous61} implies
\begin{equation}\label{eqn:Aldous62}
 \lim_{k\to\infty}\limsup_{n\to\infty}    \PR\left(\bS(\bx^{[n,k]},\by^{[n,k]},t) - \bS(\bx,\by,t)> \eps\right) = 0.
\end{equation}

Up to now, our analysis does not require $y_i>0$, but here it does. Now, let $\boldsymbol\eta = (\eta_i)\in\ell^1$ be a summable sequence of strictly positive numbers and be temporarily fixed. Let $\boldsymbol{\eta}^{[<k]} = (\eta_1,\dotms,\eta_{k-1},0,\dotms)$ be the truncation. For any $\delta>0$ and a given $k$, there exists $n_{\delta,k,\boldsymbol{\eta}} <\infty$ such that for all $n\ge n_{\delta,k,\boldsymbol{\eta}}$ 
\begin{align*}
&x_i^{(n)} \le x_i + \delta \eta_i &\textup{and}&&
&y_i^{(n)}y_{j}^{(n)}t \le y_iy_j(t+\delta) \qquad\forall i,j< k
\end{align*}
In this case, the $\xi$-coupling gives us 
\begin{equation}\label{eqn:AldousPost62}
    \bS(\bx^{(n)}, \by^{(n)},t)\le \bS(\bx^{[n,k]}+\delta\boldsymbol{\eta}^{[<k]},\by^{[n,k]},t+\delta) \qquad \forall n\ge n_{\delta,k,\boldsymbol{\eta}}.
\end{equation}
Now for any partition $\boldsymbol{\pi} = (\pi_1,\pi_2,\dotms)$ of $\N$, we can do similar analysis as~\eqref{eqn:citeLater}
\begin{align*}
    \sum_{p=1}^\infty &\left( \sum_{i\in \pi_p} z_i + \delta \eta_i 1_{[i<  k]}\right)^2
    \le \sum_{p=1}^\infty \left(\sum_{i\in \pi_p} z_i\right)^2 + 2\delta \langle \boldsymbol{\eta}, \bz\rangle^2 + \delta^2 \|\boldsymbol{\eta}\|_1^2.
\end{align*} Here we use $\|\boldsymbol{\eta}\|_2^2 \le \|\boldsymbol{\eta}\|_1^2.$
Applying this to~\eqref{eqn:AldousPost62} we see 
\begin{equation*}
    \bS(\bx^{(n)}, \by^{(n)},t)\le \bS(\bx^{[n,k]},\by^{[n,k]},t+\delta) + C_{\boldsymbol{\eta}}\|\bx^{[n,k]}\|_2^2 \delta\qquad \forall n\ge n_{\delta,k,\boldsymbol{\eta}}
\end{equation*} where the constant $C_{\boldsymbol{\eta}} = 2 \|\boldsymbol\eta\|_2^2 + \|\boldsymbol{\eta}\|_1^2\longrightarrow 0$ as $\boldsymbol{\eta}\downarrow \boldsymbol{0}$ in $\ell^1$.
In particular, we have shown
\begin{equation*}
    \limsup_{n\to\infty} \bS(\bx^{(n)},\by^{(n)}, t)  \le \bS(\bx^{[n,k]},\by^{[n,k]},t+\delta) \qquad \forall \delta>0,k \ge1.
\end{equation*}
With~\eqref{eqn:Aldous62} this implies 
\begin{equation*}
   \lim_{n\to\infty} \PR\left( \bS(\bx^{(n)},\by^{(n)}, t) > \bS(\bx,\by, t+\delta) + \eps\right) = 0
\end{equation*}

Finally, conditionally given the graph the component masses and weights, say $\bX,\bY$, of $\sW^{G}(\bx,\by,t)$ we have
\begin{align*}
    \PR&\left(\bS(\bx,\by,t+\delta) >\bS(\bx,\by,t) +\eps  \big| \bX,\bY\right) = \PR\left(\MC_2(\bX,\bY,\delta)>\|\bX\|^2 + \eps\big|\bX,\bY\right)\\
    &\le \PR\left(\MC_2(\bX+\bY, \bX+\bY, \delta) > \|\bX\|^2+\|\bY\|^2+2\langle \bX,\bY\rangle +\eps\bigg| \bX,\bY\right)\\
    &\le \frac{\delta \left(\|\bX+\bY\|^2+\eps\right) (\|\bX+\bY\|^2)}{\eps}\longrightarrow 0 
\end{align*}
as $\delta\downarrow 0$. The first inequality follows from Lemma~\ref{lem:helpLem1}. The second inequality is Lemma~\ref{lem:Lemm22Aldous}. Since $\bS(\bx,\by,t+\delta)$ decreases in $\delta$, the above limit in probability extends to convergence almost surely as $\delta\downarrow 0$. This exhibits a coupling so that~\eqref{eqn:coupledSusceptibiliy} holds and establishes the theorem.

\end{proof}

\subsection{Approximate MCMW}

A major part of the proof Theorem~\ref{thm:main2} is to show that the percolated graph $\G_n(\mu\gamma/c_n)$ is approximately the graph $\sW^G(\bX^{(n)},\bY^{(n)},\mu)$ for some random $\bX^{(n)}$, $\bY^{(n)}$ which converge weakly. To show that this approximate behavior is, in some sense, good enough, we establish the following general proposition. In order to state this, we introduce some notation.

Suppose that $G_n$, $n\ge 1$ and $G$ are graphs on the same countable vertex set. We will say that \begin{equation*}
    G_n\rightsquigarrow G
\end{equation*}
if for all vertices $u,v$
\begin{equation*}
    1_{[\{u,v\} \in E(G_n)]} \longrightarrow 1_{[\{u,v\} \in E(G)]}.
\end{equation*}

\begin{proposition}\label{prop:graphConvergence} Let $\bx^{(n)}\longrightarrow \bx^{(\infty)}$ in $\ell^2_\downarrow$ be a sequence of vertex masses. 
    Suppose that for all $1\le n\le\infty$, there are a sequence of graphs $G_n\subset G_n'$ on $\N$ where vertex $i$ has mass $x_i^{(n)}$. Denote by $\bX_n, \bX_n'$ the component masses of $G_n,G_n'$, respectively, listed in decreasing order of mass with ties broken arbitrarily. 
    
    Suppose $G_n'\rightsquigarrow G_\infty'$ and for all $i,j$
    \begin{equation*}
        \PR(i\sim  j\textup{ in }G_n) \longrightarrow \PR(i\sim j \textup{  in  } G_\infty').
    \end{equation*}
    If $\bX'_n\weakarrow\bX_\infty'$, then $\bX_n\weakarrow \bX_\infty'
    $.
\end{proposition}

\begin{proof} The proof is analogous to Lemma~\ref{lem:Lemm22Aldous}.
  Observe that for all $n\ge 1$ that $\bX_n\preceq \bX_n'$, by e.g. Lemma~\ref{lem:Lemma17Aldous}, and since $\bX'_n$ is tight in $\ell^2_\downarrow$, Corollary~\ref{cor:compactCKL} and standard tightness arguments imply
\begin{equation*}
    \left((\bX_n,\bX_n'); n \ge 1 \right) \textup{ is a tight sequence of random variables in }\ell^2_\downarrow\times\ell^2_\downarrow.
\end{equation*}
By reducing to a subsequence we can suppose that $(\bX_n,\bX_n')\weakarrow(\bX_\infty,\bX_\infty')$. We need to show $\bX_\infty = \bX'_\infty$.

Let $A_{i,j}^n$ denote the indicator that $i\sim j$ in $G_n$ and let $B_{i,j}^n$ denote the indicator that $i\sim j$ in $G_n'$. Then $A_{i,j}^n\le B_{i,j}^n \to B_{i,j}^\infty$ a.s., and for all $i\neq j$,
$$
\PR(B_{i,j}^n > A_{i,j}^{n}) = \PR(B_{i,j}^n = 1) - \PR(A_{i,j}^n = 1) \longrightarrow 0\textup{ as }n\to\infty.
$$
In particular, $A_{i,j}^n\weakarrow B_{i,j}^\infty$ for all $i,j$.

By a routine tightness argument and Skorohod's representation theorem, we can suppose that almost surely as $n\to\infty$
\begin{equation*}
    \bX_n\longrightarrow \bX_\infty, \quad\bX_n'\longrightarrow\bX'_\infty, \quad A_{i,j}^n\longrightarrow B_{i,j}^\infty\quad \textup{and}\quad B_{i,j}^n\longrightarrow B_{i,j}^\infty.
\end{equation*} 

Let us define $\bZ_n$ by
\begin{equation*}
    Z_{n,i} =\begin{cases} 0 &: \exists j<i \textup{ s.t. }A_{i,j}^n = 1\\
        x_i^{(n)} + \displaystyle \sum_{j> i} x_j^{(n)} A_{i,j}^n&: \textup{ otherwise }.
    \end{cases} 
\end{equation*}
Similarly define $\bZ_n'$ with $B_{i,j}^n$ replacing $A_{i,j}^n$ for all $1\le n\le \infty$. Note that the decreasing rearrangement of $\bZ_n$ and $\bZ_n'$ are $\bX_n$ and $\bX_n'$ respectively. Observe that Fatous' lemma implies that a.s.
\begin{equation}\label{eqn:fatouZ}
    \liminf_{n} Z_{n,i} \ge Z_{\infty,i}':= \begin{cases} 0 &: \exists j<i \textup{ s.t. }B_{i,j}^n = 1\\
        x_i^{(n)} + \displaystyle \sum_{j> i} x_j^{(n)} B_{i,j}^\infty&: \textup{ otherwise }.
    \end{cases} 
\end{equation}

We claim that 
\begin{equation}
    \label{eqn:znNormConv} \|\bZ_n\| \longrightarrow \|\bX'_\infty\|,
\end{equation} which implies $\bX_n\longrightarrow\bX_\infty'$ using~\eqref{eqn:fatouZ} (c.f.~\eqref{eqn:AldousEQN60} and equation (60) in \cite{Aldous.97}). 

To establish~\eqref{eqn:znNormConv}, we first note that 
\begin{equation*}
\|\bX_\infty\|^2 = \lim_{n}    \|\bX_n\|^2 = \lim_n \|\bZ_n\|^2 \le \lim_{n}\|\bX'_n\|^2 = \|\bX'_\infty\|^2.
\end{equation*} Now applying Fatou's lemma and~\eqref{eqn:fatouZ} we have
\begin{equation*}
    \liminf_{n\to\infty} \|\bZ_n\|^2 = \liminf_{n\to\infty} \sum_{i} Z_{n,i}^2 \ge \sum_{i} \liminf_{n} Z_{n,i}^2 \ge \sum_{i} (Z_{\infty,i}')^2 = \|\bZ_\infty'\|^2 = \|\bX_\infty'\|^2
\end{equation*} where the last inequality follows from the definition of $\bX_\infty'$. 
\end{proof}

\section{Exploration of the HCM}\label{sec:Exploration}

\subsection{Exploration of the White Configuration Model}
In this section, we describe the exploration process of the hierarchical configuration model before any of the black half-edges are paired. That is, we explore $\G_n(0)$. To simplify the notation, we omit the superscript ``$(w)$'' until it is necessary and in the algorithm below, we only refer to the white half-edges and degrees, e.g. we say choose a half-edge to mean choose a white half-edge.

Consider the following exploration of $\CM_n(\bd)$ on $n$ vertices. All the vertices $v$ and all half-edges $e$ are initially \textit{inactive} but \textit{alive}. Over the course of the exploration, we will \textit{discover}, \textit{activate}, and \textit{kill} vertices and half-edges and we will \textit{explore} vertices. The following is a re-wording of the exploration in \cite{DvdHvLS.20}.
\begin{algorithm}[{Breadth-First Exploration}]\label{alg:BFE}\,

\begin{enumerate}
    \item At step $l = 0$, all the vertices and half-edges are alive but none are active and there are no exploring vertices.
    \item For $l = 1,2,\dotsm$ we do the following:
    \begin{enumerate}
        \item If there is an \textit{exploring vertex} $v$, continue to (b). If there is an active vertex but no exploring vertex, declare the \textit{smallest} vertex to be exploring and continue to (b). Otherwise, there is no active half-edge. We choose a vertex $v$ proportional to its degree among the alive vertices (\textit{discovering} the vertex $v$); declare its half-edges to be active; declare $v$ to be exploring; then repeat step (a) with the incremented time step $l+1$. 
    \item Take an active half-edge $e$ of an exploring vertex $v$ and pair it uniformly at random to another alive half-edge $f$. Then kill both $e$ and $f$. If $f$ is incident to a vertex $u$ that is not discovered before then we discover $u$ and declare all the alive half-edges incident to $u$ (i.e. not $f$) to be active. If $f$ was the only incident half-edge to $u$ then kill $u$; but otherwise declare $u$ to be the \textit{largest} active vertex. If $e$ was the last remaining half-edge for $v$, then kill $v$. Go to (c).
    \item If all half-edges have been killed, stop the exploration. Otherwise, go back to step (a) for the next time step $l+1$.
    \end{enumerate}
\end{enumerate}
\end{algorithm}

For the vertex $i\in\CM_n(\bd)$, let $\eta_i$ denote the time that it is discovered in the exploration. Write
\begin{equation*}
    X_n(t) = -2t + \sum_{i=1}^n d_i 1_{[\eta_i \le l]}.
\end{equation*}
As noted in \cite{DvdHvLS.20} (see also \cite{Joseph.14,DvdHvLS.17}), the walk approximately encodes the component sizes of the graph. Namely, if $\cC_n(k)$ is the $k$th connected component discovered during this exploration then
\begin{equation}\label{eqn:Eccnk}
\#E(\cC_n(k)) = \tau_n(k) - \tau_n(k-1)-1,\qquad \tau_n(k) = \min\{t: X_n(t) = -2k\}.
\end{equation}

Let $\F_l$ denote the $\sigma$-algebra generated by the exploration up to time $l$ and let $\Upsilon_l$ denote the number of connected components discovered up to and including time $l$. Observe that if $\Upsilon_t> \Upsilon_{t-1}$, then a vertex $v$ was discovered in step (a) of the exploration and vertex $i$ (with $\eta_i \ge t$) was chosen with probability $\frac{d_i}{ \ell_n - 2t + 2\Upsilon_{t-1}}$. However, if we have proceeded to step (b), then a vertex $i$ with $\eta_i\ge t$ is discovered with probability $\frac{d_i}{\ell_n - 2t + 2\Upsilon_{t-1} - 1}$.
That means a.s.
\begin{equation*}
    \PR(\eta_i = t | \F_{t-1} ) 1_{[\eta_i\ge t]} = \begin{cases}\displaystyle \frac{d_i}{\ell_n - 2(l -\Upsilon_{t-1}) - 1} &: \Upsilon_t = \Upsilon_{t-1}\\\displaystyle
    \frac{d_i}{\ell_n- 2(l -\Upsilon_{t-1})} &: \Upsilon_t > \Upsilon_{t-1}
    \end{cases}
\end{equation*} Also note that $\Upsilon_t\in \F_{t-1}$ is previsible. This is because when a connected component is completely explored at time $t-1$, we know we have either completely explored the graph (i.e. all vertices have been explored) else we must find a new connected component. 
In particular, for all $t \le Tb_n$ we have
\begin{equation}\label{eqn:boundsOnEta}
     \frac{d_i}{\ell_n} \le\PR(\eta_i = t| \F_{t-1}) \le \frac{d_i}{\ell_n - 2Tb_n}\textup{  on }\{\eta_i\ge t\}.
\end{equation}

\subsection{Incorporating the Black Edges}

Let us now turn to describing how we incorporate the black half-edges. We let for each $t = 0,1,2,\dotsm$
\begin{align*}
    X_n (t) &= -2t + \sum_{i=1}^n d_i^{(w)} 1_{[\eta_i\le t]}&
    Y_n(t) &= \sum_{i=1}^n d_i^{(b)} 1_{[\eta_i\le t]}.
\end{align*}
Of course, the random variables $\eta_i$ also depend on $n$, but we omit that from the notation. 

Let $N_n(t)$ denote the number of surplus edges discovered up to time $t$ in the exploration algorithm. This only counts the number of white surplus edges in the graph $\G_n(0)$. In terms of the process $X_n(t)$ it is easy to see 
\begin{equation*}
    N_n(t) = \#\{s=0,1,\dotsm, t: X(s) = X(s-1) - 2\}. 
\end{equation*} 

For a collection of vertices $A$, write $E^{(w)}(A)$ as the collection of white edges incident to two vertices in $u,v\in A$ and write $H^{(w)} (A)$ as the collection of white half-edges incident to a vertex in $A$. Similarly, denote $E^{(b)}$ and $H^{(b)}$. Recall the definition of $\tau_n(k)$ in~\eqref{eqn:Eccnk}. 

\begin{lemma}\label{lem:HalfEdgesB}
    Let $\cC_n(k)$ denote the $k^\textup{th}$ connected component discovered in the exploration of the configuration model $\G_n(0)\overset{d}{=} \HCM_n(\bd^{(w)},\bd^{(b)})$. Then 
    \begin{equation*}
        \#E^{(w)}(\cC_n(k)) = \tau_n(k)-\tau_n(k-1) -1
    \end{equation*}
    and 
    \begin{equation*}
        \#H^{(b)}(\cC_n(k)) = Y_n(\tau_n(k)) - Y_n( \tau_n(k-1)).
    \end{equation*}
    Consequently,
    \begin{equation*}
        \#\cC_n(k) = \tau_n(k)-\tau_n(k-1) - (N_n(\tau_n(k)) - N_n(\tau_{n}(k-1))).
    \end{equation*}
\end{lemma}
\begin{proof}
    This follows form~\eqref{eqn:Eccnk}, the definition of $H^{(b)}$ and the fact that the surplus edges satisfies
    \begin{equation*}
       \#V(G) =   \#E(G) + 1 -\textup{sur}(G)
    \end{equation*}
    for any connected graph $G$.
\end{proof}

\subsection{Limit Theorems}
Recall the thinned L\'{e}vy processes $X$ and $Y$ defined in~\eqref{eqn:XinfDef}. We will also define another process $N$. 
Conditionally given $X(t)$, let $N(t)$ denote a Poisson process with intensity measure 
\begin{equation}
    \E[dN(t)] = \left(X(t) - \inf_{s\le t} X(s) \right)\,dt.\label{eqn:NinfDef}
\end{equation}

In this section, we establish the following proposition, extending \cite{DvdHvLS.20}. For a function $f:\{0,1,\dotsm\}\to \R$ we extend $f:\R_+\to\R$ and write $f(u) = f(\fl{u})$.

\begin{proposition}\label{prop:ConvergenceWalks}
    Suppose Assumption~\ref{ass:Degree}. Then following limits hold in $\D(\R_+,\R^3)$
    \begin{equation*}
        \left( \left(\frac{1}{a_n}X_n(b_nt), \frac{1}{b_n}Y_n(b_nt), N_n(b_nt)  \right); t\ge 0\right) \weakarrow \left( (X(t),Y(t), N(t)); t\ge 0\right).
    \end{equation*}
\end{proposition}

As noted in the organization, the proof follows the method in \cite{DvdHvLS.20} and for brevity we omit some parts of the proof. The remainder of this subsection is devoted to proving the joint convergence of the appropriately rescaled process $(X_n,Y_n)$.

\subsubsection{Preliminary Bound}
We will need some bounds on $\PR(\eta_j\le t)$ throughout the proof. To establish these, we appeal to the following simple lemma, whose proof is omitted.
\begin{lemma}\label{lem:boundsProb}
    Suppose that $Z\in \{1,2,\dotsm\}$ is an integer valued random variable and $\F_t$ is some filtration. Let $q_t, p_t\in[0,1]$ be constants such that
    \begin{equation*}
       q_t \le\PR(Z = t |\F_{t-1}) \le p_t \qquad\textup{ on }\{Z\ge t\}.
    \end{equation*} We also suppose that $\{Z\le  t\}\in \F_{t}$.
    Then
    \begin{equation*}
     \prod_{s=1}^t (1-p_s)  \le  \PR(Z>t) \le \prod_{s=1}^t (1-q_s)
    \end{equation*}
\end{lemma}

We apply Lemma~\ref{lem:boundsProb} above to the random variable $\eta_i$ and the filtration $\F_t$ appearing in~\eqref{eqn:boundsOnEta}. Using $e^{-(x+x^2)}\le(1-x) \le e^{-x}$ for all $x\in[0,1/2]$ we see that for $t\le Tb_n$ and $n$ large that
\begin{align}
\nonumber\exp&\left(-\left(\frac{d_i^{(w)} }{\ell_n^{(w)} - 2Tb_n}+ \frac{(d_i^{(w)})^2 }{(\ell_n^{(w)} - 2Tb_n)^2}\right)t\right)\le \left( 1-\frac{d_i^{(w)}}{\ell_n^{(w)}-2Tb_n}\right)^t   \\
\nonumber&\le \PR(\eta_i > t) \le   \left(1-\frac{d_i^{(w)}}{\ell_n^{(w)}} \right)^t\le \exp\left(-\frac{d_i^{(w)} }{\ell_n^{(w)}}t\right)
\end{align}
and so using $u-u^2 \le 1-e^{-u}\le u$ for all $u\ge 0$
\begin{equation}\label{eqn:etaIprobBounds}
   \frac{d_i^{(w)}t}{\ell_n^{(w)}}-\frac{(d_i^{(w)})^2 t^2}{(\ell_n^{(w)})^2} \le \PR(\eta_i \le t) \le \left(\frac{d_i^{(w)}}{\ell_n^{(w)} - 2Tb_n} + \frac{(d_i^{(w)})^2}{(\ell_n^{(w)} - 2Tb_n)^2} \right)t.
\end{equation}

The marginal convergence below follows from the above bounds, but the joint convergence requires a little more analysis. We refer to \cite{DvdHvLS.20} for details.
\begin{lemma}[Dhara et. al. {\cite[Lemma 9]{DvdHvLS.20}}]\label{lem:bigJumps1}
Suppose Assumption~\ref{ass:Degree}. For any $k\ge 1$, in the space $\D(\R_+,\R)^k$ the following convergence holds
\begin{equation*}
    \left( 1_{[\eta_i\le b_nt]};t\ge 0\right)_{i\in[k]} \weakarrow \left( 1_{[\xi_i\le \kappa t]};t\ge 0\right)_{i\in[k]}
\end{equation*}
where $\xi_j\sim \textup{Exp}(\theta_j)$ are independent of each other. Moreover, almost surely, the limits do not jump simultaneously.
\end{lemma}

\subsubsection{Small terms}

Observe that we can write
\begin{equation*}
    \nu_n(\lambda) = \frac{1}{\ell_n^{(w)}} \left(\sum_{i=1}^n (d_i^{(w)})^2 - \ell_n^{(w)} \right) = -1+\frac{1}{\ell_n^{(w)}}\sum_{i=1}^n (d_{i}^{(w)})^2,
\end{equation*} and so
\begin{align*}
    X_n(t) &= -2t + \sum_{i=1}^n d_i^{(w)}1_{[\eta_i\le t]}=-2t + \sum_{i=1}^n d_i^{(w)} \left(1_{[\eta_i\le t]}-\frac{d_i^{(w)}}{\ell_n^{(w)}} t\right) + \frac{t}{\ell_n}\sum_{i=1}^n (d_i^{(w)})^2\\
    &=  \sum_{i=1}^n d_i^{(w)} \left(1_{[\eta_i\le t]}-\frac{d_i^{(w)}}{\ell_n^{(w)}} t\right) + (\nu_n(\lambda)-1)t = \widetilde{X}_n(t) + A_n(t) \qquad \textup{(say)}.
\end{align*}
Clearly 
\begin{equation*}
    a_n^{-1}A_n(b_nt) = (\nu_n(\lambda)- 1)\frac{b_n}{a_n}t = \left( \lambda + o(1) \right)\frac{b_n}{a_nc_n}t =(\lambda + o(1)) t.
\end{equation*} This establishes the following lemma.
\begin{lemma}\label{lem:An(t)conv} Under Assumption~\ref{ass:Degree}, the function $A_n(t)$ defined above satisfies
    \begin{equation*}
        \left(a_n^{-1} A_n(b_nt);t\ge 0\right) \longrightarrow \left(\lambda t;t\ge 0\right),
    \end{equation*}
    locally uniformly. 
\end{lemma}

Note that under Assumption~\ref{ass:Degree} that we can find a sequence $m = m(n)\to\infty$ sufficiently slowly so that
\begin{equation}\label{eqn:m(n)defin}
\lim_{n\to\infty}   \frac{1}{n} \sum_{i>m(n)} d_i^{(b)}d_i^{(w)} = \alpha \E[D] = \alpha \kappa.
\end{equation} 

\begin{lemma}\label{lem:TailSumsConv}
Fix $\eps>0$ and $T<\infty$. Under Assumption~\ref{ass:Degree} and for the sequence $m= m(n)$ constructed above, both
    \begin{equation*}
\left(\frac{1}{b_n} \sum_{i=m+1}^n d_i^{(b)} 1_{[\eta_i\le b_nt]}; t\ge 0\right)\weakarrow (\alpha t;t\ge0).
    \end{equation*}
    and
    \begin{equation}\label{eqn:martingaleArgument}
        \left(\frac{1}{a_n}\sum_{i=m+1}^n d_i^{(w)}\left(1_{[\eta_i\le b_nt]} -\frac{d_i^{(w)}b_nt}{\ell_n^{(w)}-2Tb_n}  \right);t\ge 0\right) \weakarrow (0;t\ge 0).
    \end{equation}
\end{lemma}
\begin{proof} Since $\ell^{(b)}_n$ will not appear in this proof, omit the superscript $(w)$ and write just $\ell_n$ instead of $\ell_n^{(w)}$. 

We start with the first convergence. By Theorem 2.15 in \cite[Chapter VI]{JS.13} it is sufficient to show for all $t\ge 0$ that
\begin{equation}\label{eqn:weakConvTailY}
    \sum_{i>m} \frac{1}{b_n}d_i^{(b)} 1_{[\eta_i\le b_nt]} \weakarrow \alpha t
\end{equation}
and
\begin{equation*}
    \sum_{i>m} \left(\frac{1}{b_n}d_i^{(b)} \right)^21_{[\eta_i\le b_nt]} \longrightarrow 0.
\end{equation*} However, the second is a consequence of Assumption~\ref{ass:Degree}\eqref{enum:AssDegree3}, so we just prove~\eqref{eqn:weakConvTailY}.

Recall~\eqref{eqn:etaIprobBounds}. A consequence of that equation is that for any fixed $k$ and all $n$ such that $m(n)>k$
\begin{align*}
    \E&\left[\frac{1}{b_n}\sum_{i=m+1}^n d^{(b)}_i 1_{[\eta_i\le b_n t]}\right] \le t\sum_{i=m+1}^n d_i^{(b)}\left(\frac{d_i^{(w)}}{\ell_n-2Tb_n}  + \frac{(d_i^{(w)})^2}{(\ell_n - 2Tb_n)^2} \right)\\
    &\le t \left(\frac{\ell_n}{\ell_n-2Tb_n} + \sup_{i>k} \frac{\ell_n d_i^{(w)}}{(\ell_n-2Tb_n)^2} \right) \sum_{i=m+1}^n \frac{d_i^{(b)} d_i^{(w)}}{\ell_n}.
\end{align*}
We can use
\begin{equation}\label{eqn:boundingSupdiw}
    \sup_{i>k}d_i^{(w)} \le a_n \sup_{i>k} \left(\frac{d_i^{(w)}}{a_n}+\frac{d_i^{(b)}}{b_n} \right)  \le a_n \left(\frac{d_k^{(w)}}{a_n}+\frac{d_k^{(b)}}{b_n} \right) =  O(a_n) = o(n),
\end{equation} $\ell_n \sim (\ell_n - 2Tb_n) \sim \kappa n$, and~\eqref{eqn:m(n)defin} to see
\begin{equation}
    \label{eqn:upperBoundSec4.3}
    \limsup_{n\to\infty} \E\left[\frac{1}{b_n}\sum_{i=m+1}^n d_i^{(b)} 1_{[\eta_i\le b_nt]}\right] \le\alpha t.
\end{equation}
One can similarly get a lower bound using~\eqref{eqn:etaIprobBounds}. For $n$ large enough so that $m(n)>k$ we have
\begin{align*}
    \E&\left[\frac{1}{b_n}\sum_{i=m+1}^n d^{(b)}_i 1_{[\eta_i\le b_n t]}\right] \ge
    t\sum_{i=m+1}^n d_{i}^{(b)}\left(\frac{d_i^{(w)}}{\ell_n} - \frac{(d_i^{(w)})^2b_nt}{\ell_n^2} \right) \\
    &\ge t\left(1 - \sup_{i>k} \frac{d_i^{(w)}b_n t}{\ell_n} \right)\sum_{i=m+1}^n \frac{d_i^{(b)} d_i^{(w)}}{\ell_n}.
\end{align*}
Taking a $\liminf$ as $n\to\infty$, using $\ell_n\sim \kappa n = \kappa a_nb_n$,~\eqref{eqn:m(n)defin} and an analogous computation in~\eqref{eqn:boundingSupdiw} we see
\begin{equation*}
    \liminf_{n\to\infty} \E\left[\frac{1}{b_n}\sum_{i=m+1}^n d^{(b)}_i 1_{[\eta_i\le b_n t]}\right] \ge  \left(1- (\theta_{k}+\beta_k) t \right) \alpha t.
\end{equation*} Taking $k\to\infty$ and using the upper bound derived in~\eqref{eqn:upperBoundSec4.3} we get
\begin{equation*}
    \lim_{n\to\infty} \E \left[ \frac{1}{b_n} \sum_{i=m+1}^n d_i^{(b)} 1_{[\eta_i\le b_nt]}\right] = \alpha t.
\end{equation*}

It is established in \cite[equation (4.12)]{DvdHvLS.20} that the random variables $1_{[\eta_i\le t]}$ for $i\in[n]$ are negatively correlated. Since $\Var(aI) \le a \E[aI]$ for any Bernoulli random variable, this negative correlation implies for any fix $k$ and all $m(n)>k$
\begin{align*}
    \Var&\left(\frac{1}{b_n} \sum_{i=m+1}^n d_i^{(b)} 1_{[\eta_i\le b_nt]}\right) \le \sum_{i=m+1}^n \Var\left(\frac{d_i^{(b)}}{b_n}1_{[\eta_i\le b_nt]}\right) \le \sum_{i=m+1}^n \frac{d_i^{(b)}}{b_n} \E\left[\frac{d_i^{(b)}}{b_n}1_{[\eta_i\le b_nt]}\right] \\
    &\le \left(\frac{d_k^{(w)}}{a_n}+\frac{d_k^{(b)}}{b_n}\right) \E\left[\frac{1}{b_n} \sum_{i=m+1}^n d_i^{(b)} 1_{[\eta_i\le b_nt]}\right], 
\end{align*} 
and so
\begin{equation*}
    \limsup_{n\to\infty}\Var\left(\frac{1}{b_n} \sum_{i=m+1}^n d_i^{(b)} 1_{[\eta_i\le b_nt]}\right)  \le  \left(\theta_k+\beta_k\right) \alpha t.
\end{equation*}
As $\theta_k+\beta_k\to0$ as $k\to\infty$, $\Var\left(\frac{1}{b_n} \sum_{i=m+1}^n d_i^{(b)} 1_{[\eta_i\le b_nt]}\right)\longrightarrow 0$. Chebyshev now implies~\eqref{eqn:weakConvTailY}.

Let us write $M_n(t) = \frac{1}{a_n}\sum_{i=m+1}^n d_i^{(w)}\left(1_{[\eta_i\le t]} -\frac{d_i^{(w)}t}{\ell_n^{(w)}-2Tb_n} \right)$ where we do not scale the time yet. It is shown between equation (4.5) and (4.9) in \cite[pg 1526]{DvdHvLS.20} that $M_n$ is an $\F_l$ super-martingale where $\F_l$ is the $\sigma$-algebra generated by the graph exploration and they argue therein that $\sup_{t\le Tb_n} |M_n(t)| \weakarrow 0$ as $n\to\infty$ whenever $m$ is kept fixed. It does not immediately follow that the same holds whenever $m = m(n)\to\infty$; however, their proof carries over. We repeat it in the appendix.
\end{proof}

\subsubsection{Dominant Terms}

Observe that a direct consequence of Lemma~\ref{lem:bigJumps1} is that jointly over the index $i$ we have in $\D(\R_+,\R^2)$
\begin{equation*}
\left(\left(\frac{d_i^{(w)}}{a_n} \left(1_{[\eta_i\le b_nt]} - \frac{d_i^{(w)}}{\ell_n} b_nt\right), \frac{d_i^{(b)}}{b_n} 1_{[\eta_i\le b_nt]}\right);t\ge 0\right)\weakarrow \left(\left(\theta_i 1_{[\xi_i\le \kappa t]} - \frac{\theta_i^2}{\kappa}t, \beta_i 1_{[\xi_i\le \kappa t]}\right);t\ge 0\right).
\end{equation*} Note that the limits above for any two $i\neq j$ will have distinct jump times a.s. Therefore, using standard properties about addition in the Skorohod space (see \cite{Whitt.80} or Proposition 2.2(a)\cite[Chapter VI]{JS.13}) we have for each $k$ fixed 
\begin{align*}
    &\left( \sum_{i=1}^k \frac{d_i^{(w)}}{a_n} \left(1_{[\eta_i\le b_nt]} - \frac{d_i^{(w)}}{\ell_n} b_nt\right),\sum_{i=1}^k \frac{d_i^{(b)}}{b_n} 1_{[\eta_i\le b_nt]} \right)_{t\ge 0}\\
    &\qquad\weakarrow\left(\sum_{i=1}^k \theta_i\left( 1_{[\xi_i\le \kappa t]}-\frac{\theta_i}{\kappa}t \right), \sum_{i=1}^k \beta_i 1_{[\xi_i\le \kappa t]} \right)_{t\ge 0}
\end{align*}
in $\D(\R_+,\R^2)$.

Therefore, by Theorem 3.2 in \cite{Billingsley.99}, the joint re-scaled convergence of $(X_n,Y_n)$ stated in Proposition~\ref{prop:ConvergenceWalks} is established once we prove the following lemma.
\begin{lemma}
    For any $\eps>0$ and any $T<\infty$
    \begin{equation*}
        \lim_{k\to\infty} \limsup_{n\to\infty} \PR\left(\sup_{t\le T} \sum_{i=k+1}^{m(n)}  \left|\frac{d_i^{(w)}}{a_n} \left(1_{[\eta_i\le b_nt]} - \frac{d_i^{(w)}}{\ell_n} b_nt\right) \right|>\eps \right) = 0
    \end{equation*}
    and
      \begin{equation*}
        \lim_{k\to\infty} \limsup_{n\to\infty}  \PR\left(\sup_{t\le T} \sum_{i=k+1}^{m(n)}\left|  \frac{d_i^{(b)}}{b_n} 1_{[\eta_i\le b_nt]} \right|>\eps \right) = 0.
    \end{equation*}
\end{lemma}
\begin{proof}
    The first statement follows from Lemma~\ref{lem:TailSumsConv} and equation (4.17) in \cite{DvdHvLS.20}. 
    
    To see that the second holds, we note that by~\eqref{eqn:etaIprobBounds} and~\eqref{eqn:boundingSupdiw} we get
    \begin{align*}
        \limsup_{n\to\infty}&\, \E\left[\sum_{i=k+1}^{m(n)}  \frac{d_i^{(b)}}{b_n} 1_{[\eta_i\le b_nT]}\right] \le \limsup_{n} \sum_{i=k+1}^{m(n)} \frac{d_i^{(b)}d_i^{(w)}}{\ell_n}\left(1+\frac{d_k^{(w)}+d_k^{(b)}}{\ell_n}\right)T\\
        & = \frac{1}{\kappa}\sum_{i=k+1}^\infty \theta_i\beta_i T.
    \end{align*} This converges to zero as $k\to\infty$ as $\langle\btheta,\bbeta\rangle<\infty$. The result follows by Markov's inequality and the fact that the supremum occurs at $t = T$.
\end{proof}

\subsubsection{Putting Things Together}

We have shown
\begin{lemma}
     Suppose Assumption~\ref{ass:Degree}. Then following limits hold in $\D(\R_+,\R^2)$
    \begin{equation*}
        \left( \left(\frac{1}{a_n}X_n(b_nt), \frac{1}{b_n}Y_n(b_nt)  \right); t\ge 0\right) \weakarrow \left( (X(t),Y(t)); t\ge 0\right).
    \end{equation*}
\end{lemma}

Going from the above lemma to Proposition~\ref{prop:ConvergenceWalks} is non-trivial as the increments of $N_n$ are precisely the times $t$ such that $X_n(t) = X_{n}(t-1)-2$. However, one can overcome this by defining $X_n'$ and $Y_n'$ by deleting the corresponding increments and showing that the scaled distance in the $J_1$ topology doesn't grow too large. As a formal proof of Proposition~\ref{prop:ConvergenceWalks} using the above lemma is exactly the same argument as Lemma 16 in \cite{DvdHvLS.20}, we omit it.

\subsection{Large Components are Explored Early}

We now turn to analyzing the number of black half-edges incident to the large components. In order to establish Theorem~\ref{thm:main1}, we will need some technical results that (informally) say that not only are all large components of $\G_n(0)$ explored early but so are the components of $\G_n(0)$ that have a lot of black half-edges incident to their vertices. In order to state this, we let $\cC_{n}^\circ (i)$ denote the connected component containing the vertex $i$ in $\G_n(0)$ if $i$ is its minimal vertex, otherwise set $\cC^\circ_n(i) = \emptyset$.

\begin{proposition}\label{prop:bigcomps}
    Suppose Assumption~\ref{ass:Degree}. Then for any $\eps>0$
\begin{align*}
   \lim_{K\to\infty} \limsup_{n\to\infty} \PR\left(\sum_{i>K} \left(\#\cC_{n}^\circ(i)\right)^2> \eps b_n^2 \right) = 0
\end{align*}
and
\begin{equation*}
       \lim_{K\to\infty} \limsup_{n\to\infty} \PR\left(\sum_{i>K} \left(\#H^{(b)}(\cC_{n}^\circ(i))\right)^2> \eps b_n^2 \right) = 0.
\end{equation*}
\end{proposition}

We first recall the following lemma of \cite{DvdHvLS.20} (Lemma 35) which itself is an extension of Lemmas 5.1 and 5.2 in \cite{Janson.10}.
\begin{lemma}[Dhara et. al. {\cite[Lemma 35]{DvdHvLS.20}}]\label{lem:susceptibility1}
    Suppose that $\bd^{(w)}$ and $\bd^{(b)}$ are degree sequences of length $n$ and $\nu:= \frac{1}{\ell_n^{(w)}} \sum_{i=1}^n d_i^{(w)}(d_i^{(w)}-1) <1$. Independently of the graph $\CM(\bd^{(w)})$, let $U$ be a uniformly chosen vertex and let $V$ denote a vertex chosen according to some distribution. Denote by $C(v)$ the component containing the vertex $v$. Then
    \begin{equation*}
        \E\left[\sum_{i\in C(V)} d_i^{(b)}\right] \le \E\left[d_{V}^{(b)}\right] + \frac{ \displaystyle \E[d_V^{(w)}]  \E[ d_U^{(w)} d_U^{(b)}]}{\displaystyle \E[d^{{(w)}}_U](1-\nu_n) }.
    \end{equation*}
\end{lemma}
Now that we have the above lemma, we can turn to the proof of the proposition.

\begin{proof}[Proof of Proposition~\ref{prop:bigcomps}] 
The proof closely follows the proof of Lemma 12 in \cite{DvdHvLS.20}. Since that result implies the statement above involving $\#\cC_n^\circ(i)$, we only prove the second statement. 

Let us fix a $K>0$ arbitrary but large and denote by $\G^{[K]}_n$ the graph obtained by deleting all white half-edges attached to $1,2,\dotsm,K\in \G_n(0)\sim \CM_n(\bd^{(w)})$. As the remaining half-edges in $\G_n(0)$ were still connected using a uniform matching, the resulting graph $\G^{[K]}_n$ is a configuration model with a \textit{random} degree distribution $\widetilde\bd = (\widetilde{d}_1,\dotms, \widetilde{d}_n)$, where $\widetilde{d}_i = 0$ for all $i\le K$ and $\widetilde{d}_i\le d_i^{(w)}$ for all $i$. 

Denote by $\nu_n^{[K]}$ the criticality parameter (first equality below). Let $\lambda' = \lambda+1$ and note that by Assumption~\ref{ass:Degree}\eqref{enum:AssDegree4} that for all large $n$, $\nu_n \le 1+\lambda' c_n^{-1}$. Therefore, for all large $n$ we can bound $\nu_n^{[K]}$ as follows:
\begin{align*}
    \nu_n^{[K]} &= \frac{\sum_{i=1}^n \widetilde{d}_i (\widetilde{d}_i-1)}{\sum_{i=1}^n \widetilde{d}_i} \le \frac{\sum_{i=1}^n d_i^{(w)}(d_i^{(w)}-1) - \sum_{i=1}^{K} d_i^{(w)}(d_i^{(w)}-1)}{\ell_n^{(w)} - 2\sum_{i=1}^K d_i^{(w)}}\\
    &= \frac{\ell_n}{\ell_n -2 \sum_{i=1}^K d_i^{(w)}}\frac{\sum_{i=1}^n d_i^{(w)}(d_i^{(w)}-1) - \sum_{i=1}^K d_i^{(w)}(d_i^{(w)}-1) }{\ell_n}  \\
    &\le \frac{1}{1-\frac{2}{\ell_n} \sum_{i=1}^K d_i^{(w)}} \cdot \left(\nu_n - \frac{1}{\ell_n} \sum_{i=1}^K (d_i^{(w)})^2 + \frac{1}{\ell_n} \sum_{i=1}^K d_i^{(w)} \right).
    \end{align*}
    Note that for any $K$, for all large $n$ we have $d_i^{(w)}< \frac{1}{2} (d_i^{(w)})^2$ for all $i\le K$ using Assumption~\ref{ass:Degree}\eqref{enum:AssDegree2}. Also note $\frac{1}{1-x} \le 1+2x$ for all $x\in[0,1/2]$. Therefore, we get
\begin{align*}
    \nu^{[K]}_n&\le \left(1+\frac{4}{\ell_n}\sum_{i=1}^{K}d_i^{(w)} \right)\left(\nu_n - \frac{1}{2\ell_n} \sum_{i=1}^K (d_i^{(w)})^2 \right).
\end{align*}

Since $d_i^{(w)}/a_n\to \theta_i$ it holds that for all large $n$ and all $i\le K$ that $\frac{1}{\sqrt{2}}a_n\theta_i \le  d_i^{(w)}\le 2a_n\theta_i$. Therefore
\begin{align*}
    \nu^{[K]}_n&\le \left(1+\frac{8}{\ell_n}\sum_{i=1}^K \theta_i a_n\right)\left(1+\frac{\lambda'}{c_n} - \frac{1}{4\ell_n} \sum_{i=1}^K \theta_i^2 a_n^2\right)\le 1+\frac{\lambda'}{c_n} - \frac{a_n^2}{4\ell_n} \sum_{i=1}^K \theta_i^2  + \frac{8a_n}{\ell_n}\sum_{i=1}^K \theta_i  + \frac{8\lambda'a_n}{c_n\ell_n}\sum_{i=1}^K \theta_i.
\end{align*}

Using $a_n^2/n = a_n/b_n = 1/c_n$, $a_n/n = 1/(b_nc_n)$, and $\ell_n\sim \kappa n$ for $\kappa>0$ we can turn the above bound into the following for all large $n$ and for $\lambda'' = \lambda +2 = \lambda'+1$ that
\begin{align*}
    \nu_n^{[K]}&\le 1+ \frac{1}{c_n} \left(\lambda' - \frac{1}{8\kappa} \sum_{i=1}^K \theta^2_i + \left(\frac{8}{\kappa b_n}+\frac{8\lambda'}{\kappa b_nc_n}\right)\sum_{i=1}^K \theta_i   \right)\le 1+ \frac{1}{c_n} \left(\lambda'' - \frac{1}{8\kappa} \sum_{i=1}^K \theta_i^2\right).
\end{align*} 
As $K\to\infty$, $\sum_{i=1}^\infty \theta_i^2 \to \infty$ and so we can find a $K_0$ sufficiently large so that for all $K\ge K_0$ it holds that $\nu_n^{[K]}<1$ for all $n$ large.

Let us condition on $\tilde{\bd}$. We apply Lemma~\ref{lem:susceptibility1} to the graph $\G^{[K]}_n|\widetilde\bd \sim \CM_n(\widetilde\bd)$ with $\PR(V = i) = \frac{d_i^{(b)}}{\ell_n^{(b)}}$ and $U$ uniformly chosen (both independent of the configuration model and the degree sequence $\widetilde\bd$) to get
\begin{align}\label{eqn:cluster1.1}
    \E\left[\sum_{i\in C(V)} d_i^{(b)} \bigg| \widetilde\bd\right] \le \E\left[d_{V}^{(b)}\right] + \frac{ \displaystyle \E[\widetilde{d}_{V}|\widetilde\bd] \E[\widetilde{d}_U d^{(b)}_U|\widetilde\bd]}{\E[\widetilde{d}_U |\widetilde\bd](1-\nu_n^{[K]})}.
\end{align}
Note the following hold for all $n$
\begin{align*}
\E[d_V^{(b)}] &= \frac{1}{\ell_n^{(b)}}\sum_{i=1}^n (d_i^{(b)})^2 &
    \E[\widetilde{d}_V|\widetilde\bd]  &= \frac{1}{\ell_n^{(b)}}\sum_{i=1}^n \tilde{d}_i d_i^{(b)} \le \frac{1}{\ell_n^{(b)}}\sum_{i=K+1}^n d_i^{(w)} d_i^{(b)} \\
    \E\left[\widetilde{d}_U d^{(b)}_U|\widetilde\bd\right] &= \frac{1}{n} \sum_{i=1}^n \tilde{d}_id_i^{(b)} \le \frac{1}{n}\sum_{i=K+1}^n d_i^{(w)}d_i^{(b)}
\end{align*}
and the following hold for all $n$ sufficiently large
\begin{align*}
    \E\left[\tilde{d}_U|\widetilde\bd\right]&= \frac{1}{n}\sum_{i=1}^n \tilde{d}_i = \frac{1}{n}\left(\ell_n^{(w)} - 2\sum_{i=1}^{K} d_i^{(w)} \right) \ge \kappa - \frac{4}{b_n}\sum_{i=1}^K \theta_i\ge \frac{\kappa}{2}\\
    1-\nu_n^{[K]} &\ge \frac{1}{c_n} \left( \frac{1}{8\kappa}\sum_{i=1}^K \theta_i^2-\lambda''\right).
\end{align*}
Using the above bounds and~\eqref{eqn:cluster1.1}, we have for all large $n$
\begin{align*}
    \E\left[\sum_{i\in C(V)} d_i^{(b)} \bigg| \widetilde\bd\right] \le \frac{1}{\ell_n^{(b)}}\sum_{i=1}^n (d_i^{(b)})^2 + \frac{16 c_n}{ \sum_{i=1}^K \theta_i^2 - 8\lambda''\kappa }\frac{1}{\ell_n^{(b)} n}\left(\sum_{i=K+1}^n d_i^{(w)}d_i^{(b)}\right)^2,
\end{align*}
and the right-hand side does not depend on $\widetilde\bd$.
Now using Assumption~\ref{ass:Degree}\eqref{enum:AssDegree5}, we know that $\ell_n^{(b)} = \Theta(n)$ and $n^{-1}\sum_{i=1}^n (d_i^{(b)})^2 = \Theta(1) = o(c_n)$. In particular, by Assumption~\ref{ass:Degree}\eqref{enum:AssDegree3}
\begin{equation*}
    \limsup_{n\to\infty} \frac{1}{\ell_n^{(b)} n} \left(\sum_{i=K+1}^n d_i^{(w)}d_i^{(b)}\right)^2 \le C_1
\end{equation*} for some constant $C_1$ independent of $K$ and 
\begin{equation*}
    \limsup_{n\to\infty} \frac{1}{\ell_n^{(b)} c_n} \sum_{i=1}^{n} (d_i^{(b)})^2 \left(\sum_{i=1}^K \theta_i^2 -8\lambda''\kappa \right) =0.
\end{equation*} Therefore we have for all $K$ large enough so that $\sum_{i=1}^K \theta_i^2 -8\lambda'' \kappa>0$, it holds that for all sufficiently large $n$
\begin{align}
  \nonumber\E &\left[\sum_{i\in C(V)} d_i^{(b)} \bigg| \widetilde\bd\right] \le \frac{1}{\ell_n^{(b)}}\sum_{i=1}^n (d_i^{(b)})^2 + \frac{16 c_n}{ \sum_{i=1}^K \theta_i^2 - 8\lambda''\kappa }\frac{1}{\ell_n^{(b)} n}\left(\sum_{i=K+1}^n d_i^{(w)}d_i^{(b)}\right)^2\\
    &\le \frac{(16 C_1 + 1)}{\sum_{i=1}^{K }\theta_i^{2} -8\lambda'' \kappa} c_n = \frac{C_2}{\sum_{i=1}^{K }\theta_i^{2} -8\lambda'' \kappa} c_n.\label{eqn:boundCn1}
\end{align} Here $C_2$ is a constant independent of $K$, but the bound holds for all $n\ge n_0(K)$.

Now for the graph $\G^{[K]}_n$, let $\cC^{\circ,K}(j)$ denote the connected components listed in order of their size and let $\#H^{(b)}(\cC^{\circ,K}(j))$ denote the number of black half-edges for in the larger graph $\G_n(0)$. It holds that
\begin{equation*}
    \sum_{j\ge 1} \left(\# H^{(b)}(\cC^{\circ,K}(j))\right)^2=  n \E\left[ \sum_{i\in C(V)} d_i^{(b)} \bigg |\widetilde\bd\right] \le \frac{ C_2 n c_n}{\sum_{i=1}^K \theta_i^2 - 8\lambda'' \kappa}.
\end{equation*}
Now taking expectations and noting $b_n^2 = c_n n$ we get for all large $n$
\begin{equation*}
    \E\sum_{j\ge 1} \left(\# H^{(b)}(\cC^{\circ,K}(j))\right)^2 \le \frac{C_2}{\sum_{i=1}^K \theta_i^2 - 8\lambda'' \kappa} nc_n =  \frac{C_2}{\sum_{i=1}^K \theta_i^2 - 8\lambda'' \kappa} b_n^2.
\end{equation*}
Hence Markov's inequality tells us for all large $K$ and all $\eps>0$ that
\begin{equation*}
    \limsup_{n\to\infty} \PR\left( \sum_{j\ge 1} \left(\# H^{(b)}(\cC^{\circ,K}(j))\right)^2  >\eps b_n^2\right) \le \frac{C_2}{\eps} \frac{1}{\sum_{i=1}^K \theta_i^2 - 8\lambda''\kappa}. 
\end{equation*}
Since $\btheta\notin\ell^2$ by Assumption~\ref{ass:Degree}\eqref{enum:AssDegree2}, we get for all $\eps>0$
\begin{equation*}
    \lim_{K\to\infty} \limsup_{n\to\infty} \PR\left( \sum_{j\ge 1} \left(\# H^{(b)}(\cC^{\circ,K}(j))\right)^2  >\eps b_n^2\right)  \le \lim_{K\to\infty} \frac{C_2}{\eps} \frac{1}{\sum_{i=1}^K \theta_i^2 - 8\lambda''\kappa} = 0.
\end{equation*}
The proof is finished by noting that
\begin{equation*}
    \sum_{i>K}\left(\# H^{(b)}(\cC^{\circ}(j))\right)^2\le_{\textup{st}}\sum_{j\ge 1} \left(\# H^{(b)}(\cC^{\circ,K}(j))\right)^2 
\end{equation*}
where $X\le_{\text{st}}Y$ means that $X$ is stochastically dominated by $Y$, i.e. $\PR(X>t) \le \PR(Y>t)$ for all $t>0$.  
\end{proof}

The above proposition tells us something about the connected components discovered late as well. To state this result, we write
\begin{equation*}
    \cC_{n,\ge T} (i) = \begin{cases}
        \emptyset &: \exists v\in \cC_{n}(i)\text{ s.t. }v\text{ discovered before time }Tb_n\\
        \cC_n(i)&: \textup{else}.
    \end{cases}
\end{equation*}
\begin{lemma} \label{lem:LateComponentsAllSmall}
    Suppose Assumption~\ref{ass:Degree}.
    \begin{equation*}
        \lim_{T\to\infty} \limsup_{n\to\infty} \PR\left(\sum_{i} (\#\cC_{n,\ge T} (i))^2 + (\#H^{(b)}(\cC_{n,\ge T}(i)))^2 > \eps b_n^2 \right) =0
    \end{equation*}
\end{lemma}

\begin{proof}
    For each $K$ it holds that
   \begin{align}
     \nonumber \PR&\left(\sum_{i} (\#\cC_{n,\ge T} (i))^2 + (\#H^{(b)}(\cC_{n,\ge T}(i)))^2 > \eps b_n^2\right)\\
     \nonumber &\le \PR\left(\sum_{i} (\#\cC_{n,\ge T} (i))^2 + (\#H^{(b)}(\cC_{n,\ge T}(i)))^2 > \eps b_n^2,\text{ and } \eta_i\le Tb_n \forall i\in[K]\right)\\
    \nonumber  &\qquad \qquad+ \PR(\exists i\in[K]\textup{ s.t. } \eta_i> Tb_n)\\
      \label{eqn:Hbig1} &\le  \PR\left(\sum_{i>K}\left(\# \cC^{\circ}(j)\right)^2 > \frac{1}{2}\eps b_n^2\right) + \PR\left(\sum_{i>K}\left(\# H^{(b)}(\cC^{\circ}(j))\right)^2 > \frac{1}{2}\eps b_n^2\right)\\
    \nonumber  &\qquad \qquad+\PR(\exists i\in[K]\textup{ s.t. } \eta_i> Tb_n)
    \end{align} 
    By Lemma~\ref{lem:bigJumps1} we have
    \begin{equation}\label{eqn:Hbig2}
        \limsup_{n\to\infty } \PR(\exists i\in[K]\text{ s.t. }\eta_i>T b_n) \le \lim_{n\to\infty} \sum_{i=1}^K \PR(\eta_i>Tb_n) = \sum_{i=1}^K \exp(-\theta_i \kappa T).
    \end{equation}
Combining equations~\eqref{eqn:Hbig1} and~\eqref{eqn:Hbig2} we get for all $K$
    \begin{align*}
         \limsup_{T\to\infty}&\, \limsup_{n\to\infty} \PR\left(\sum_{i} (\#\cC_{n,\ge T} (i))^2 + (\#H^{(b)}(\cC_{n,\ge T}(i)))^2 > \eps b_n^2\right)\\
         &\le  \limsup_{T\to\infty} \limsup_{n\to\infty} \PR\left(\sum_{i>K}\left(\# H^{(b)}(\cC^{\circ}(j))\right)^2 > \frac{1}{2}\eps b_n^2\right) + \limsup_{T\to\infty} \limsup_{n\to\infty} \PR\left(\sum_{i>K}\left(\# \cC^{\circ}(j)\right)^2 > \frac{1}{2}\eps b_n^2\right) \\
         &\qquad+ \limsup_{T\to\infty} \limsup_{n\to\infty} \PR(\exists i\in[K]\textup{ s.t. } \eta_i> Tb_n)\\
         &= \limsup_{n\to\infty} \PR\left(\sum_{i>K}\left(\# H^{(b)}(\cC^{\circ}(j))\right)^2 > \frac{1}{2}\eps b_n^2\right) + \limsup_{n\to\infty} \PR\left(\sum_{i>K}\left(\# \cC^{\circ}(j)\right)^2 > \frac{1}{2}\eps b_n^2\right).
    \end{align*}
    By Proposition~\ref{prop:bigcomps}, these last term can be made arbitrarily small. 
\end{proof}

\section{Convergence of Component Sizes}\label{sec:goodFunctions}

\subsection{Good functions}

In this section we establish and recall some results about the $J_1$ topology and convergence of excursions. For an interval $I\subset\R_+$ we let $\D_+(I)$ the collection of c\`adl\`ag functions $f$ such that $f(t)\ge f(t-)$ for all $t\in I$. Recall that $\EE(f)$ is the collection of excursion intervals of the function $f$. Following \cite{DvdHvLS.20}, define $\cY(f) = \{r: (l,r)\in \EE(f)\}$. 
\begin{definition}[Good function on {$[0,T]$}]\label{def:Good[0,T]} We say that a function $f\in \D_+([0,T])$ is \textit{good (on $[0,T]$)} if the following three properties hold
\begin{enumerate}[(i)]
    \item For all $\cY(f)$ does not have any isolated points;
    \item $[0,\sup \cY(f)]\setminus \bigcup_{e\in \EE(f)} e$ has Lebesgue measure $0$;
    \item \label{enum:good[0,t]:iii}$f$ does not attain a local minimum at any point $r\in \cY(f)$.
\end{enumerate}
\end{definition}

The following result is established in \cite[Remark 12]{DvdHvLS.20}.
\begin{lemma}\label{lem:good1}
    Suppose that $f\in \D_+([0,T])$ is good. Then $f$ is continuous at each $r\in \cY(f)$.
\end{lemma}

The following lemma is also easy to see.
\begin{lemma}
    Suppose that $f\in \D_+([0,T])$ is good and let $(l_j,r_j)\in \EE(f)$ for $j\in[2]$ with $l_1<l_2$. Then $r_1<l_2$. In other words, the excursions are not nested.
\end{lemma}
\begin{proof}
    Since $l_1<l_2$ and $(l_1,r_1)$ is an excursion interval, we have \begin{equation*} f (r_1) = \inf _{s\le r_1} f(s) = \inf_{s\le l_1} f(s) \ge \inf_{s\le l_2} f(s),
    \end{equation*}
    where the first equality is valid by Lemma~\ref{lem:good1}. If $r_1 \ge l_2$, then we would have equality above, implying that $r_1$ is a local minimum of $f$ contradicting Definition~\ref{def:Good[0,T]}\eqref{enum:good[0,t]:iii}. 
\end{proof}

\begin{definition}[Good on {$\R_+$}]\label{def:GoodonR}
A function $f\in \D_+ = \D_+(\R_+)$ is \textit{good} on $\R_+$ if the following hold:
\begin{enumerate}[(i)]
    \item $f$ is good on $[0,T]$ for each $T>0$;
    \item For all $\eps>0$, the set $\EE(f)$ has only finitely many excursions of length $r-l>\eps$;
    \item $f$ does not possess an infinite length excursion.
\end{enumerate}
\end{definition}

We now turn to a key lemma which is a simple extension of \cite[Lemma 7]{Aldous.97}. The result of Aldous holds for continuous limits.
\begin{lemma}\label{lem:goodPP}
    Suppose that $f:\R_+\to\R$ is good. Let $\Xi^{(\infty)} = \{(r,r-l);(l,r)\in\EE(f)\}$. Suppose that $f_n\to f$ in the $J_1$ topology as $n\to\infty$ and suppose that for each $n$ there is a sequence $(t_{n,i};i\ge 0)$ such that 
    \begin{enumerate}[(i)]
        \item \label{enum:5.5:i}$0=t_{n,0}<t_{n,1}<\dotsm$ and $\lim_{i\to\infty} t_{n,i} = +\infty$ for all $n$;
        \item \label{enum:5.5:ii} $f_n(t_{n,i}) = \inf_{s\le t_{n,i}} f_n(s)$ for all $i$ and all $n$;
        \item \label{enum:5.5:iii} $\max_{i: t_{n,i}\le s_0} f_{n}(t_{n,i})-f_{n}(t_{n,i+1}) \longrightarrow 0$ as $n\to\infty$ for each $s_0<\infty$.
    \end{enumerate}
    Write $\Xi^{(n)} = \{(t_{n,i}, t_{n,i}-t_{n,i-1});i\ge 1\}$. Then $\Xi^{(n)}\to\Xi^{(\infty)}$ in the topology of vague convergence of the associated counting measures on $\R_+\times(0,\infty)$.
\end{lemma}
\begin{proof}
    It suffices to show the following. Suppose that $K = [a,b]\times [c,d]\subset \R_+\times(0,\infty)$ is a compact rectangle with $\Xi^{(\infty)} \cap \partial K = \emptyset$, then for all $(r,r-l)\in \Xi^{(\infty)}\cap K$, there exists a sequence $i_n = i_n(r)$ such that for all large $n$ $(t_{n,i_n-1},t_{n,i_n}-t_{n,i_n-1})\in K$ and
    \begin{equation}\label{eqn:tninconv}
        t_{n,i_n-1}\to l\qquad\textup{and}\qquad t_{n,i_n}\to r,
    \end{equation}
    as well as $\#(\Xi^{(n)}\cap K) \longrightarrow \#(\Xi^{(\infty)}\cap K)$. 

    Before continuing to the proof, it is useful to recall that that $\inf_{s\le t}f(s)$ is continuous, and 
    \begin{equation}\label{eqn:infLocUnif}
        \inf_{s\le u} f_n(s) \longrightarrow \inf_{s\le u} f(s)\qquad\textup{ locally uniformly in }u.
    \end{equation}
    
    First we show~\eqref{eqn:tninconv}. To do this let us fix an excursion $(l,r)$ with $(r,r-l)\in K$. Let $t\in(l,r)$ be a point of continuity of $f$. Note that since $t$ is a continuity point it holds that
    \begin{equation}\label{eqn:fn(t)converges}
        \lim_{n\to\infty} f_n(t) = f(t).
    \end{equation}

    Set $i_n = \min\{i: t_{n,i}>t\}$. We claim 
    \begin{equation}\label{eqn:limSupTnin}
       \liminf_{n\to\infty} t_{n,i_n-1} \ge l\qquad\textup{ and }\qquad \limsup_{n\to\infty} t_{n,i_n} \le r.
    \end{equation} To see this note that by hypothesis~\eqref{enum:5.5:ii} and $t_{n,i_n}\ge t$, we have
    \begin{equation*}
        f_n(t_{n,i_n}) = \inf_{s\le t_{n,i_n}}f_n(s) \le \inf_{s\le t} f_n(s) \longrightarrow \inf_{s\le t} f(s) =  f(r)
    \end{equation*}
    and
    \begin{equation*}
         f_n(t_{n,i_n-1}) = \inf_{s\le t_{n,i_n-1}}f_n(s) \ge \inf_{s\le t} f_n(s) \longrightarrow f(r).
    \end{equation*}
    So
    \begin{equation*}
        \limsup_{n\to\infty} f_n(t_{n,i_{n}}) \le f(r)\qquad\textup{and}\qquad \liminf_{n\to\infty} f_n(t_{n,i_n-1}) \ge f(r).
    \end{equation*}
    Since $t_{n,i_n-1}\le t$ for all $t$, we can use hypothesis~\eqref{enum:5.5:iii} to conclude that
    \begin{equation*}
        f_{n}(t_{n,i_n}) \ge f_n(t_{n,i_n-1})-o(1)
    \end{equation*} and so
    \begin{equation}\label{eqn:infTninTof(r)}
      f_{n}(t_{n,i_n}) =   \inf_{s\le t_{n,i_n}} f_n(s)\to f(r) = \inf_{s\le r} f(s) . 
    \end{equation}
    But if there were some $\delta>0$ such that $t_{n,i_n}> r+\delta$ for infinitely many $n$, then
    \begin{equation*}
        \liminf_{n\to\infty} \inf_{s\le t_{n,i_n}} f_n(s) \le \liminf_{n\to\infty} \inf_{s\le r+\delta} f_n(s) = \inf_{s\le r+\delta} f_n(s)\longrightarrow \inf_{s\le r+\delta} f(s)< \inf_{s\le r}f(r)
    \end{equation*}
    where we used~\eqref{eqn:infLocUnif} to take the limit and Definition~\ref{def:Good[0,T]}\eqref{enum:good[0,t]:iii} to conclude the strict inequality. This contradicts~\eqref{eqn:infTninTof(r)}. Hence the right-hand side of~\eqref{eqn:limSupTnin} holds. A similar argument holds for the left-hand side.

    We now show that
    \begin{equation*}
        \liminf_{n\to\infty} t_{n,i_n} \ge r.
    \end{equation*} By~\eqref{eqn:limSupTnin} we see that $\sup_{n} t_{n,i_n}$ is bounded, and therefore there exists a $M$ sufficiently large such that for all $n$, $t_{n,i_n}\in[t,M]$. Let $r'$ be any subsequential limit of $(t_{n,i_n})$. First, by~\eqref{eqn:limSupTnin}, $r'\in[t,r]$. Also for all $s\in[t,r)$ it holds that $f(s-)>f(r)$ and so it suffices to show that $f(r'-)\le f(r)$. Then for the relevant subsequence denoted by $n'$, by Proposition 2.1 (b.1) and (b.3) in \cite[Chapter VI]{JS.13} we have
    \begin{equation*}
        f(r'-) \le \liminf_{n'} f_{n'}(t_{n',i_{n'}}).
    \end{equation*} But by~\eqref{eqn:infTninTof(r)}, we know that $\lim_{n'} f_{n'}(t_{n',i_{n'}}) = f(r)$ giving us the desired result. A similar argument holds to show that $\limsup_{n\to\infty} t_{n,i_n} \le l$ and we have thus established the existence of some $i_n$ such that~\eqref{eqn:tninconv} holds. Finally, note that for all $n$ large enough it must be the case that $t_{n,i_{n}}\in[a,b]$ and $t_{n,i_{n}}-t_{n,i_n-1}\in [c,d]$ as $r-l\in (c,d)$ and $r\in(a,b)$ since $\Xi^{(\infty)}$ does not intersect $\partial K$. 

    To show that $\#(\Xi^{(n)}\cap K)\to \#(\Xi^{(\infty)}\cap K)$, it suffices to show that for any subsequence $n'$ and any convergent sequence $(x_{n'},y_{n'})\in \Xi^{(n')}\cap K$ must satisfy $\lim_{n'} (x_{n'},y_{n'})\in \Xi^{(\infty)}\cap K$ as we have already shown $\liminf_{n} \#(\Xi^{(n)}\cap K) \ge \#(\Xi^{(\infty)}\cap K)$. 
    
    Let us denote the subsequence by $n$ and suppose, for the sake of contradiction, that there is some sequence $j_n$ such that for all $n$
    \begin{equation*}
        (t_{n,j_n}, t_{n,j_n}-t_{n,j_n-1})\in \Xi^{(n)}\cap K
    \end{equation*}
    but
    \begin{equation*}
        (t_{n,j_n}, t_{n,j_n}-t_{n,j_n-1}) = (v,v-u)\notin \Xi^{(\infty)}\cap K.
    \end{equation*} Clearly, we must have $v\in[a,b]$ and $v-u\in [c,d]$. To establish the contradiction, we show $(v,v-u)\in \Xi^{(\infty)}$. That means $(u,v)$ is actually an excursion interval of $f$.
   
    As $c>0$, we must have $v>u$. By~\eqref{eqn:infLocUnif} we can use both hypotheses~\eqref{enum:5.5:ii} and~\eqref{enum:5.5:iii} as we did in the first part of the proof, to conclude that
    \begin{equation*}
        \lim_{n\to\infty} f_n(t_{n,j_n}) = \inf_{s\le v} f(s) =\lim_{n\to\infty} f_n(t_{n,j_n-1}) = \inf_{s\le u} f(u).
    \end{equation*} Since $f$ is good $(u,v)\subseteq (l,r)$ for some uniquely defined excursion interval $(l,r)$. We claim $v = r$ and $u = l$.

     To see that, note that since $(u,v)\subseteq (l,r)$, we have 
     \begin{equation*}
       \lim_{n\to\infty} f_n(t_{n,j_n}) = \lim_{n\to\infty} f_n(t_{n,j_n-1}) = f(r-) = f(r),
     \end{equation*} where we also used Lemma~\ref{lem:good1}. By Proposition 2.1 (b.1) and (b.3) in \cite[Chapter VI]{JS.13} implies
     \begin{equation*}
     \inf_{s\le v} f(s) = \liminf_{n\to\infty} f_n(t_{n,j_n}) \ge f(v-).
    \end{equation*} Which implies that $f(v-) = \inf_{s\le v} f(s)$. Similarly, $f(u-) \ge \inf_{s\le u} f(s)$. This implies that $(u,v)$ are an excursion interval since the excursion intervals of $f$ cannot be nested, this implies that $(u,v) = (l,r)$, the desired contradiction. This establishes the lemma.
\end{proof}

\subsection{Jointly good functions}

We now extend this to convergence of pairs of functions $(f,g).$ Let $\D_\uparrow (I)$ denote the collection of non-decreasing c\`adl\`ag functions on an interval $I$. Let $\D_{+,\uparrow}(I)\subset \D(I)^2$ denote the collection of functions $F = (f,g):I\to \R^2$ such that $f\in \D_+(I)$ and $g\in \D_\uparrow(I)$. 
\begin{definition}[Jointly good function]\label{def:jointlygoodCompact} We will say that the pair $(f,g)\in \D_{+,\uparrow}([0,T])$ are \textit{jointly good on }$\R_+$ if $f$ is good on $\R_+$ and $g$ is continuous at each endpoint of an excursion of $f$. That is $g$ is continuous at each $t \in \bigcup_{(l,r)\in \EE(f)} \{l,r\}$.
\end{definition}

The following is a simple extension of Lemma~\ref{lem:goodPP}, we only sketch the proof.
\begin{lemma}\label{lem:goodPPwithG}
Suppose that $(f_n,g_n)\to (f,g)$ in the product $J_1$ topology and that $(f,g)$ are jointly good on $\R_+$. Suppose $(t_{n,i};i\ge 0)$ is as in Lemma~\ref{lem:goodPP}. Then
\begin{align*}
    \Xi^{(n)}&:= \left\{(t_{n,i},t_{n,i}-t_{n,i-1}, g_n(t_{n,i})-g_n(t_{n,i-1}));i\ge 1\right\}\\
    &\longrightarrow\Xi^{(\infty)} :=\{(r,r-l,g(r)-g(l)):(l,r)\in \EE(f)\}
\end{align*}
in the vague topology for the associated counting measures on $\R_+\times (0,\infty)\times \R_+$.
\end{lemma}
\begin{proof}[Proof Sketch]
Proposition VI.2.1(b.5) in \cite{JS.13} states that for c\`adl\`ag functions $g_n\to g$ in the $J_1$ topology then for any continuity point $t$ of $g$ and any sequence $t_n\to t$ it holds that $g_n(t_n)\to g(t)$ and $g_n(t_n-)\to g(t)$. Therefore, if $t_{n,i_n}\to r$ and $t_{n,i_n-1}\to l$, since $g$ is continuous at both $l, r$ it holds that $g_n(t_{n,i_n})\to g(r)$ and $g_n(t_{n,i_n-1})\to g_n(l).$
\end{proof}


\subsection{Goodness of $(X,Y)$}

The following lemma is needed to apply Lemma~\ref{lem:goodPPwithG}.
\begin{lemma}\label{lem:XYisGood}
    The process $(X,Y)$ of Proposition~\ref{prop:ConvergenceWalks} are jointly good on $\R_+$.
\end{lemma}

The proof follows from the proceeding lemmas which generalizations of the well-known result of Aldous and Limic \cite[Proposition 14]{AL.98}
\begin{lemma}[Dhara et. al. {\cite[Lemma 15]{DvdHvLS.20}}]\label{lem:Zproperties}
Let $X$ be as in~\eqref{eqn:XinfDef}. Under Assumption~\ref{ass:Degree} the process $X$ is good on $\R_+$.
\end{lemma}
\begin{lemma}[Broutin, Duquesne, Wang \cite{BDW.22}]\label{lem:excursions}
    Let $(X,Y)$ be as in~\eqref{eqn:XinfDef}. Suppose that $\btheta\in \ell^3\setminus\ell^2$ and $\langle \btheta,\bbeta\rangle<\infty$. Then a.s. both
    \begin{equation*}
        X(t) = X(t-)\textup{ and } Y(t) = Y(t-)\textup{ for all }t\in \bigcup_{(l,r)\in \EE(X)} \{l,r\}
    \end{equation*}
    and 
    \begin{equation*}
   \forall s<t:   \quad  \inf_{u\le s} X(u) = \inf_{u\le t} X(u) \qquad\text{if and only if}\qquad (s,t)\subset (l,r)\in \EE(X).
    \end{equation*}

\end{lemma}
\begin{proof}
    We only need to show the statement for $X$ since the jump times of $Y$ are contained in those of $X$. But by Lemma 5.7(iii) in \cite{BDW.22} implies that
\begin{equation*}
\{X(t)> \inf_{s\le t} X(s)\} = \bigcup_{e\in \EE(X)} e\qquad\textup{ a.s. }.
\end{equation*} This is the first displayed equation. The second is Lemma 5.7(iv) in \cite{BDW.22}.
\end{proof}

\subsection{Weak Convergence of Point Processes}

We now turn to extending the convergence of $\Xi^{(n)}\to\Xi^{(\infty)}$ into convergence in some $\ell^{2}$-sense. 
\begin{proposition}\label{prop:WeakConvPointProcess}
Let $\Xi^{(n)}\subset \R_+\times(0,\infty)\times \R_+$, for $1\le n\le \infty$, are locally finite point sets as in Lemma~\ref{lem:goodPPwithG} for some random functions $(F_n,G_n)$. Suppose that
\begin{enumerate}[(i)]
\item\label{enum:5.11.i} $(F_n,G_n)\weakarrow (F,G)$ where $(F,G)$ are a.s. jointly good.
    \item \label{enum:5.11.ii}  For all $1\le n\le \infty$, we can enumerate the atoms of $\Xi^{(n)}$ by
\begin{equation*}
    \Xi^{(n)} = \{ (t_i^{(n)},x_i^{(n)},y_i^{(n)});i\ge1\}
\end{equation*} where $\{t_i^{(n)};i\ge 1\}$ are all distinct,  $i\mapsto x_i^{(n)}$ is non-increasing with $i$, and if $x_i^{(n)} = x_j^{(n)}$, then $t_i^{(n)}< t_j^{(n)}$;
\item \label{enum:5.11.iii}For all $\eps>0$
\begin{equation}\label{eqn:tailSumBoundProp5.10}
    \lim_{T\to\infty} \limsup_{n\to\infty}\PR\left( \sum_{i: t_i^{(n)}> T} (x_i^{(n)})^2+(y_i^{(n)})^2 > \eps\right) = 0.
\end{equation}
\item \label{enum:5.11.iv}Almost surely the Stieljes measure $G(dt)$ does not charge the excursion end points of $F$; i.e. 
    \begin{equation*}
        \int_{0}^\infty  1_{[F(t) = \inf_{s\le t} F(s) ]} \, G(dt) = 0.
    \end{equation*}

\end{enumerate}
If $x_1^{(\infty)}>x_2^{(\infty)}>\dotsm > 0$ a.s., then 
\begin{equation}\label{eqn:l2pointprocess}
    \left((x_i^{(n)},y_i^{(n)});i\ge 1\right) \weakarrow \left((x_i^{(\infty)}, y_i^{(\infty)});i\ge 1\right)\qquad\textup{in }\ell^{2,2}.
\end{equation}

\end{proposition}
\begin{proof}
    By Skorohod's representation theorem, we can suppose that $(F_n,G_n)\to (F,G)$ a.s. By Lemma~\ref{lem:goodPPwithG}, this implies $\Xi^{(n)}\to \Xi^{(\infty)}$ a.s. as well. Then the convergence in~\eqref{eqn:l2pointprocess} holds with respect to the product topology.
    
    To establish the result in the $\ell^{2,2}$ topology, it suffices to show that $\left\{(x_i^{(n)},y_i^{(n)})_{i\ge 1}\right\}_{n\ge 1}$ are tight in $\ell^{2,2}$. To show this, it suffices to show
    \begin{equation}\label{eqn:normbounded}
        \lim_{M\to\infty} \limsup_{n\to\infty} \PR\left( \sum_{i} (x_i^{(n)})^2+(y_i^{(n)})^2 > M\right) = 0,
    \end{equation}
    and for all $\eps>0$
    \begin{equation}\label{eqn:tailTight}
        \lim_{\delta\downarrow 0} \limsup_{n\to\infty} \PR\left( \sum_{i: x_i^{(n)}\le \delta} (x_i^{(n)})^2 + (y_i^{(n)})^2 > \eps\right) = 0.
    \end{equation}

    First, let $T>0$ be arbitrary. Then
    \begin{align*}
        \sum_{i} (x_i^{(n)})^2 &+ (y_i^{(n)})^2 = \sum_{i: t_{i}^{(n)}\le T}(x_i^{(n)})^2 + (y_i^{(n)})^2 + \sum_{i: t_i^{(n)}>T}(x_i^{(n)})^2 + (y_i^{(n)})^2\\
        &\le \left(\sum_{i: t_{i}^{(n)}\le T} x_i^{(n)} + y_i^{(n)} \right)^2 + \sum_{i: t_i^{(n)}>T}(x_i^{(n)})^2 + (y_i^{(n)})^2\\
        &\le \left(T+G_n(T)\right)^2+\sum_{i: t_i^{(n)}>T}(x_i^{(n)})^2 + (y_i^{(n)})^2.
    \end{align*}
    In the first inequality above we used the fact that if $a_i\ge 0$ and $a = \sum_{i} a_i$ then $\sum_{i} a_i^2 \le a^2$. It follows that 
    \begin{equation*}
        \PR\left(\sum_{i} (x_i^{(n)})^2 + (y_i^{(n)})^2 > M+1\right) \le \PR\left( (G_n(T)+T)^2 > M\right) + \PR\left(\sum_{i: t_i^{(n)}>T}(x_i^{(n)})^2 + (y_i^{(n)})^2>1\right).
    \end{equation*}
    Since $G_n\weakarrow G$, for all but countably many $T$ we have $G_n(T)\weakarrow G(T)$ and so 
    \begin{equation*}
        \limsup_{n\to\infty} \PR((T+G_n(T))^2 > M)  \le \PR( (T+G(T))^2 \ge M).
    \end{equation*}
    But then we have for all but countably many $T$,
    \begin{align*}
        \limsup_{M\to\infty} &\limsup_{n\to\infty} \PR\left(\sum_{i} (x_i^{(n)})^2 + (y_i^{(n)})^2 > M\right)\\
        &\le \limsup_{M\to\infty} \limsup_{n\to\infty} \PR\left( (G_n(T)+T)^2 > M\right) + \limsup_{n\to\infty} \PR\left(\sum_{i: t_i^{(n)}>T}(x_i^{(n)})^2 + (y_i^{(n)})^2>1\right)\\
        & = \lim_{M\to\infty} \PR((G(T)+T)^2 > M)+ \limsup_{n\to\infty} \PR\left(\sum_{i: t_i^{(n)}>T}(x_i^{(n)})^2 + (y_i^{(n)})^2>1\right)\\
        &=\limsup_{n\to\infty} \PR\left(\sum_{i: t_i^{(n)}>T}(x_i^{(n)})^2 + (y_i^{(n)})^2>1\right).
    \end{align*} Now apply~\eqref{eqn:tailSumBoundProp5.10} to get~\eqref{eqn:normbounded}.

    To get~\eqref{eqn:tailTight} we apply a similar argument. First, for any $\eps>0$ and $\eta>0$ we can find some $T_\eta$ such that
    \begin{align*}
        \limsup_{n\to\infty} \PR\left(\sum_{i: t_i^{(n)}> T_\eta} (x_i^{(n)})^2+(y_i^{(n)})^2 > \eps \right) \le \eta.
    \end{align*}
    It follows that     
    \begin{equation*}
        \limsup_{\delta\downarrow 0} \limsup_{n\to\infty} \PR\left(\sum_{\substack{i: t_i^{(n)}> T_\eta, x_i^{(n)}\le \delta}} (x_i^{(n)})^2+(y_i^{(n)})^2 > \eps \right) =\limsup_{n\to\infty} \PR\left(\sum_{\substack{i: t_i^{(n)}> T_\eta}} (x_i^{(n)})^2+(y_i^{(n)})^2 > \eps \right)  \le \eta.
    \end{equation*}
    So~\eqref{eqn:tailTight} follows once we establish that for Lebesgue a.e. $T\in(0,\infty)$
    \begin{equation*}
        \limsup_{\delta\downarrow 0} \limsup_{n\to\infty} \PR\left(\sum_{\substack{i: t_i^{(n)}\le T, x_i^{(n)}\le \delta}} (x_i^{(n)})^2+(y_i^{(n)})^2 > \eps \right) = 0.
    \end{equation*}
    Fix $T<\infty$ such that $F(T) > \inf_{s\le T}F(s)$ and $F(T) = F(T-)$ a.s. Since $(F,G)$ are good, and the lengths of the excursion intervals of $F$ are a.s. distinct, we can label the excursions of $F$ that end before time $T$ by $((l_p,r_p);p\ge 1)$ in the a.s. unique way so that $r_1-l_1> r_2-l_2>\dotsm$. Also, almost surely there is some unique excursion interval of $F$ which $T$ straddles and we call $V$ the \textit{left} endpoint of this excursion interval. Observe
    \begin{equation*}
     \sum_{p=1}^\infty (r_p-l_p) = V \qquad \textup{ and }\qquad \sum_{p=1}^\infty (G(r_p)-G(l_p)) = G\left(V\right).
    \end{equation*} Indeed, the second summation uses hypothesis~\eqref{enum:5.11.iv} and 
    \begin{equation*}
        \sum_{p=1}^\infty \left(G(r_p)-G(l_p)\right) = G\left(\bigcup_{p=1}^\infty (l_p,r_p]\right) = G\left([0,\sup_{p} r_p]\right) = G(V),
    \end{equation*} where we write $G$ for both the function and the Stieljes measure. As there is a bijective correspondence between $(l_p,r_p)_{p\ge 1}$ and $t_i^{(\infty)}\le T$, we have
    \begin{equation}\label{eqn:deltaHelp2}
    \lim_{\delta\downarrow 0} \sum_{i: t_i^{(\infty)}\le T, x_i^{(\infty)}\le \delta} x_i^{(\infty)}+ y_i^{(\infty)} = 0\qquad 
    \textup{a.s.}\end{equation} 
    Since $\Xi^{(n)}\to \Xi^{(\infty)}$ in the vague topology, it is not hard to see that for a.e. $\delta>0$
    \begin{equation}\label{eqn:deltaHep1}
        \sum_{i: t_i^{(n)} \le T, x_i^{(n)}\ge \delta} x_i^{(n)} + y_i^{(n)} \longrightarrow \sum_{i:  t_i^{(\infty)} \le T, x_i^{(\infty)}\ge \delta} x_i^{(\infty)} + y_i^{(\infty)}.
    \end{equation} One needs to be careful to bound the $x_i^{(n)}$ and $y_i^{(n)}$ as well and work on a compact set, but as $\limsup_{n\to\infty} x_i^{(n)} \le T$ and $\limsup_{n\to\infty} y_i^{(n)} \le G(T)$ for all $t_{i}^{(n)}\le T$ this easy to do. 

    Write $V_n = \max\{t_i^{(n)}: t_i^{(n)}\le T\}$. Note that $V_n\to V$ a.s. as we saw in Lemma~\ref{lem:goodPP}. 
    But for the $\delta$ such that~\eqref{eqn:deltaHep1} holds we have
    \begin{align*}
        \PR&\left(\sum_{i: t_i^{(n)} \le T, x_i^{(n)}< \delta} (x_i^{(n)})^2 + (y_i^{(n)})^2 >\eps\right)\le \PR\left(\left(\sum_{i: t_i^{(n)} \le T, x_i^{(n)}< \delta} x_i^{(n)} + y_i^{(n)}\right)^2>\eps \right)\\
        &= \PR\left(\left(V_n+G_n(V_n) - \sum_{i: t_i^{(n)} \le T, x_i^{(n)}\ge \delta} (x_i^{(n)} + y_i^{(n)})\right)^2>\eps \right)\\ 
        &\longrightarrow \PR\left(\left(V+G(V) - \sum_{i: t_i^{(\infty)} \le T, x_i^{(\infty)}\ge \delta} (x_i^{(\infty)} + y_i^{(\infty)})\right)^2>\eps\right)\\
        &= \PR\left(\left(\sum_{i: t_i^{(\infty)} \le T, x_i^{(\infty)}< \delta} (x_i^{(\infty)} + y_i^{(\infty)})\right)^2>\eps\right).
    \end{align*} To take the limit we use that $V_n\to V$ a.s. and $G$ is continuous at $V$ so $G_n(V_n)\to G(V)$ (see Proposition 2.1(b) in \cite[Chapter VI]{JS.13}). We also use the convergence in~\eqref{eqn:deltaHep1}. Taking $\delta\downarrow 0$, and using~\eqref{eqn:deltaHelp2} gives~\eqref{eqn:tailTight}.   
\end{proof}

\subsection{Proof of Theorem~\ref{thm:main1}}
The proof of Theorem~\ref{thm:main1} is a more-or-less straightforward application of Proposition~\ref{prop:WeakConvPointProcess}. 

Set $J_n(k) = \min\{p = 0,1,\dotms: p+N_n(p) = k\}$. Observe that
\begin{equation*}
  \check{\tau}_n(k):=  \min\{t: X_n\circ J_n(t) = -2k\} = \tau_n(k) + N_n(\tau_n(k)),
\end{equation*} 
where $\tau_n(k)$ is defined in~\eqref{eqn:Eccnk}. Set $\check{X}_n(t) = X_n\circ J_n(t)$ and $\check{Y}_n(t) = Y_n\circ J_n(t)$. By Proposition~\ref{prop:ConvergenceWalks} and a result of Whitt \cite{Whitt.80} on first passage times it holds that
\begin{equation*}
    \left(\frac{1}{b_n}J_n(b_nt);t\ge 0\right) \weakarrow (t;t\ge 0),
\end{equation*} jointly with the convergence in Proposition~\ref{prop:ConvergenceWalks}. Therefore, $F_n(t) = a_n^{-1} \check{X}_n(b_nt)$, $G_n(t) = b_n^{-1}\check{Y}_n(b_nt)$ satisfy
\begin{equation*}
    \left(F_n,G_n\right)\weakarrow (X,Y)
\end{equation*} where $X,Y$ are as in Proposition~\ref{prop:ConvergenceWalks} and the convergence is in $\D(\R_+)^2$.

Apply Proposition~\ref{prop:WeakConvPointProcess} with $F_n,G_n$ and 
\begin{equation*}
    t_{n,i} = \frac{1}{b_n}\check\tau_n(i)\qquad\forall i.
\end{equation*} 
Then the $j^\text{th}$ component discovered $\cC_n^\circ(j)$, say, satisfies
\begin{align*}
    \frac{1}{b_n}& \# \cC_n^\circ(j) = t_{n,j}-t_{n,j-1}  &\textup{and}&&
    \frac{1}{b_n} &\#H^{(b)}(\cC_n^\circ(j)) = G_n(t_{n,j}) - G_n (t_{n,j-1}),
\end{align*}
by Lemma~\ref{lem:HalfEdgesB}. Lemma~\ref{lem:XYisGood} implies hypothesis~\eqref{enum:5.11.i} of Proposition~\ref{prop:WeakConvPointProcess}, Lemma~\ref{lem:LateComponentsAllSmall} implies hypothesis~\eqref{enum:5.11.iii}, and hypothesis~\eqref{enum:5.11.ii} is simply a rearrangement assumption. It is shown that $x_i^{(\infty)}$ are distinct a.s. in \cite[Fact 1, pg 1544]{DvdHvLS.20}.
The verification of hypothesis~\eqref{enum:5.11.iv} follows from the observation that the Lebesgue decomposition of the random Stieljes measure $Y$ is
\begin{equation*}
    Y(dx) = \alpha dx + \sum_{i=1}^\infty b_i \delta_{\xi_i/\kappa} (dx) = \alpha dx + \sum_{0<t<\infty} (Y(t)-Y(t-))\,\delta_t(dx)
\end{equation*} As $\Leb(\{t: X(t) = \inf_{s\le t} X(s)\}) = 0$ as $X$ is good (Lemma~\ref{lem:Zproperties}), and since $Y$ does not jump on $\{t: X(t) = \inf_{s\le t} X(s)\}$ by Lemma~\ref{lem:excursions} we have verified~\eqref{enum:5.11.iv}.

\section{Inhomogeneous Percolation of the HCM}\label{sec:Inhomo}

Recall that we view elements $\bw\in \ell^{2,2}_+$ as an element of $\ell^2\times\ell^2$ by setting $\bw = (\bx,\by)$ where $\bw = ((x_i,y_i);i\ge 1)$ and $\bx:=(x_i;i\ge1)$ and $\by = (y_i;i\ge 1)$.

Before continuing we define the limit of Theorem~\ref{thm:main2}. Set $(\bX,\bY) = \Gamma^\downarrow(X,Y)$ where $X,Y$ are the limits in Theorem~\ref{thm:main1}. Almost surely $(\bX,\bY)\in \ell^2\times\ell^2$, and under \eqref{eqn:dBconve} $\bY\in(0,\infty)^\infty$. Recall that
\begin{equation}\label{eqn:bzetaDef}
\bzeta^{\HCM} (\lambda,\mu) := \MC_2(\bX,\bY,\mu).
\end{equation}
As this will be useful later, we will write $\bX_n$ and $\bY_n$ for the vectors 
\begin{equation*}
    \bX_n =(\#\cC_n(j);j\ge 1),\qquad\textup{and}\qquad \bY_n = (\#H^{(b)}(\cC_n(j)))
\end{equation*}
where $\cC_n(j)$ are the connected components of $\G_n(0)$ listed in order of decreasing cardinality.

\subsection{Connection to Percolated HCM}

Consider the graph $\G_n(0)$. We denote the $Q_1^{(b)}(s)$ the number of black-half edges that are yet to be paired in $\G_n(s)$. That is
\begin{equation*}
    Q_n^{(b)}(s) = E^{(b)}(\G_n(s)) - \frac{1}{2} H^{(b)}(\G_n(s)).
\end{equation*}
By standard properties of Poisson processes, we can see that the family $\G_n(s)$ can be constructed as follows.
\begin{algorithm}[Dynamic construction of $\G_n(s)$]\,\label{alg:Dynamic}
\begin{enumerate}
    \item At time $s = 0$, we have the graph $\G_n(0)$ and $2Q_n^{(b)}(0)$ many unpaired black half-edges. Independently of $\G_n(0)$ let $\Xi_n(t)$ denote an inhomogeneous Poisson process with rate $Q_n^{(b)}(t)\,dt$.
    \item At each jump time $t_k$ of $\Xi_n(t)$ choose two distinct unpaired black half-edges uniformly at random and pair them. This decreases $Q_n^{(b)}(t)$ by one.
    \item The graph $\G_n(s)$ is the graph $\G_n(0)$ along with all paired black edges up to time $s$.
\end{enumerate}
\end{algorithm}

Observe that we have the following merging dynamics. Initially, we have some connected components with (integer) masses $X_{n,i}(0) = \#\cC_n(i)$ and (integer) weights $Y_{n,i}(0) = \#H^{(b)}(\cC_n(i))$. Independently for every pair of blocks $i$ and $j$ existing at time $t$, they merge into a single block of mass $X_{n,i}(t) + X_{n,j}(t)$ and weight $Y_{n,i}(t)+Y_{n,j}(t)-2$ at rate
\begin{align*}
  Q_n^{(b)}(t)& \left( \frac{Y_{n,i}(t)}{2Q_n^{(b)}(t)}\frac{Y_{n,j}(t)}{2Q_n^{(b)}(t)-1}+  \frac{Y_{n,i}(t)}{2Q_n^{(b)}(t)}\frac{Y_{n,j}(t)}{2Q_n^{(b)}(t)-1}\right)=  \frac{Y_{n,i}(t)Y_{n,j}(t)}{2Q_n^{(b)}(t)-1}
\end{align*}
and block $i$ is transformed into a new block of the same mass but with weight $Y_{n,i}(t)-2$ with rate
\begin{equation*}
      Q_n^{(b)}(t)\frac{Y_{n,i}(t)}{Q_n^{(b)}(t)}\frac{Y_{n,i}(t)-1}{2Q_n^{(b)}(t-)-1} = \frac{1}{2}\frac{Y_{n,i}(t)(Y_{n,i}(t)-1)}{2Q_{n}^{(b)}(t)-1}.
\end{equation*} At least for small time $t$, the above quantities should be order $\frac{\Theta(b_n^2)}{\Theta(n)} = \Theta(c_n)$ under Assumption~\ref{ass:Degree} by Theorem~\ref{thm:main1}. Most of the remainder of the section is devoted to turning this into a formal argument.

\subsection{Modified Graph Process}

There are two main issues at present. The first is that each time the system changes we loose two black half-edges, and second the speed depends on the total number of black half-edges in the system. To overcome this issue, we use the same idea of \cite{DvdHvLS.20} and \cite{BDvdHS.20} and introduce a \textit{modified graph process} which we label $\overline{\G}_n(s).$ 
\begin{algorithm}[Modified Graph Process]\,\label{alg:Modified}

\begin{enumerate}
    \item 
    At time $s = 0$, the graph $\overline{\G}_n(0) = \G_n(0)$. 
    \item At rate $Q_n^{(b)}(0)$ a uniformly chosen pair of black half-edges are selected and an edge is included between the two incident vertices; however, the two black half-edges remain. 
    \item At time $s$ the graph $\overline{\G}_n(s)$ consists of all the new black edges that are formed. 
    \end{enumerate}
\end{algorithm}

 The following lemma is clear by Poisson thinning and examining Step 2 in both Algorithm~\ref{alg:Dynamic} and Algorithm~\ref{alg:Modified}. We let $\ord:\ell^2_+\to \ell^2_\downarrow$ denote the natural ordering map.
\begin{lemma}\label{lem:UpperCoupling}
    One can couple the modified graph process and the graph process such that for all $s\ge 0$ 
    \begin{equation*}
        \G_n(s)\subset \overline{\G}_n(s).
    \end{equation*}
    In particular, 
    \begin{align*}
    \ord&(\#\cC: \cC\textup{ is a connected component of }\G_n(s))\\&\preceq \ord(\#\cC: \cC\textup{ is a connected component of }\overline{\G}_n(s)).
    \end{align*}
\end{lemma}

Observe that for the modified process two components $i$ and $j$ of masses $X_{n,i}(t)$ and $X_{n,j}(t)$ (resp.) and weights $Y_{n,i}(t)$ and $Y_{n,j}(t)$ (resp.) merge at rate
\begin{equation*}
    \frac{Y_{n,i}(t)Y_{n,j}(t)}{2Q_{n}^{(b)}(0)-1}
\end{equation*}
into a new component of mass $X_{n,i}(t)+X_{n,j}(t)$ and weight $Y_{n,i}(t) + Y_{n,j}(t)$. That is the merging dynamics are precisely those of the MCMW \textit{slowed down }by a factor of $2Q_{n}^{(b)}(0)-1$. This establishes:

\begin{lemma}\label{lem:modified=MC}
    Conditionally given $\overline{\G}_n(0)$, the sizes of the connected components of $\overline{\G}_n(s)$   satisfy
    \begin{align*}
        \bigg(\ord&\left(\#\cC\subset\overline{\G}_n(s): \cC\textup{ connected component of }\overline{\G}_n(s)\right);s\ge0 \bigg)\bigg| \overline{\G}_n(0)\\
        &\overset{d}{=} \left( \MC_2\left(\mathbf{X}_n,\mathbf{Y}_n; \frac{s}{2Q_n^{(b)}(0)-1}\right);s\ge 0\right)\\
    \end{align*}
\end{lemma}

We now turn to the convergence result.

\begin{lemma}\label{lem:CONVERGE2}
Suppose Assumption~\ref{ass:Degree} and~\eqref{eqn:dBconve} hold. Then for $s = s_n = \mu\gamma_n c_n^{-1}$ for some sequence $\gamma_n\to \gamma$
\begin{equation*}
    \frac{1}{b_n}\ord\left(\#\cC\subset\overline{\G}_n(s): \cC\textup{ connected component of }\overline{\G}_n(s)\right) \weakarrow \bzeta^{\HCM}(\lambda,\mu)\qquad \textup{ in } \ell^{2}_\downarrow
\end{equation*} where $\bzeta^\HCM$ is defined in~\eqref{eqn:bzetaDef}.
\end{lemma}
\begin{proof} Since we only deal with $Q_n^{(b)}(0)$ in the proof, we write $Q_n^{(b)} = Q_n^{(b)}(0)$.

By Theorem~\ref{thm:main1} and Skorohod's representation theorem, we can work on a probability space such that $(\bX_n,\bY_n)\longrightarrow \Gamma^\downarrow(X,Y)$ a.s. in $\ell^{2,2}.$ This implies that $\frac{1}{b_n}\bX_n\to \bX$ and $\frac{1}{b_n}\bY_n\to \bY$ in $\ell^2$ for $(\bX,\bY) = \Gamma^\downarrow(X,Y).$

It is easy to see that
\begin{equation*}
    \MC_2(a\bx,b\by, ct) \overset{d}{=} a\MC_2(\bx, b\sqrt{c} \by ;t).
\end{equation*} 
By Lemma~\ref{lem:modified=MC} and the above equality in law, we get
    \begin{align*}
        \frac{1}{b_n}&\ord\left(\#\cC\subset\overline{\G}_n\left(s\right): \cC\textup{ connected component}\right) \\
        &\overset{d}{=} \MC_2\left(\frac{1}{b_n} \bX_n, \bY_n, \frac{\mu\gamma_n}{ c_n} \frac{1}{2{Q_n^{(b)}-1}} \right)\overset{d}{=} \MC_2\left(\frac{1}{b_n} \bX_n, \sqrt{\frac{{\gamma_n}}{c_n(2 Q_n^{(b)}-1)}}\,\bY_n, \mu \right).
    \end{align*}
    Note that $b^2_n =c_n n$ and so~\eqref{eqn:dBconve} implies $\frac{\gamma_n}{c_n (2Q_n^{(b)}-1)} \sim \frac{1}{b_n^2}$ and so 
    \begin{equation*}
       \sqrt{\frac{{\gamma_n}}{c_n(2 Q_n^{(b)}-1)}}\,\bY_n\longrightarrow \bY.
    \end{equation*} Also $\frac{1}{b}_n\bX_n\to \bX$. Applying Theorem~\ref{thm:MC} gives the stated convergence result.
\end{proof}

\subsection{Asymptotics of $Q_n^{(b)}(t)$}

We wish to apply Proposition~\ref{prop:graphConvergence} and we have gathered many of the necessary ingredients already. Lemma~\ref{lem:UpperCoupling} gives us the desired coupling of $G_n' = \overline{\G}_n(\mu\gamma/c_n)$ and $G_n = \G_n(\mu\gamma/c_n)$. Lemma~\ref{lem:modified=MC} gives us the convergence of the component masses of $G_n'$. In order to apply Proposition~\ref{prop:graphConvergence}, we need to establish some lower bounds on the edge probabilities in $\G_n(s)$ for small times $s$. 

\begin{lemma}\label{lem:QnSize}
Suppose Assumption~\ref{ass:Degree} and~\eqref{eqn:dBconve} hold. For any $T>0$ fixed, it holds that for all $n$ sufficiently large that for any sequence $\eps_n = \frac{\omega(n)}{c_n}\to0$ that $\omega(n)\to\infty$
\begin{equation*}
    \PR\left(\frac{1}{n} Q_n^{(b)}(0) - \frac{1}{n} Q_n^{(b)}(Tc_n^{-1}) >\eps_n \right) \le8 \gamma T \frac{c_n}{\omega(n)^2 n}
\end{equation*}
and for any $\delta_n\to 0$
\begin{equation*}
    \PR\left(\frac{1}{n}Q_n^{(b)}(Tc_n^{-1}) > \frac{1}{n}Q_n^{(b)}(0) e^{-T/c_n} + \delta_n\right)\le \frac{2\gamma T}{\delta_n^2 nc_n}.
\end{equation*}
\end{lemma}
\begin{proof} We use the differential equation method \cite{Warnke.19,Wormald.99} and follow the proof of Lemma 7.2 in \cite{BDvdHS.20}. Recall that $Q_n^{(b)}(t)$ decreases by 1 at rate $Q_n^{(b)}(t)dt$ and so we can realize $Q_n^{(b)}$ as the solution to the time-change equation
\begin{equation*}
    Q_n^{(b)}(t) = Q_n^{(b)}(0) - \mathcal{N}\left(\int_0^t Q_n^{(b)}(s)\,ds \right)
\end{equation*} where $\mathcal{N}$ is Poisson process on $\R_+.$ Since $\mathcal{N}(s)= M(s)-s$ for a martingale $M$ with quadratic covariation $\langle M\rangle (s) = s$, we have
\begin{equation*}
    Q_n^{(b)}(t) = Q_n^{(b)}(0) - \mathcal{M}(t) - \int_0^t Q_n^{(b)}(s)\,ds
\end{equation*}
where $\mathcal{M}(t) = M(\int_0^t Q_n^{(b)}(s))$ is a martingale w.r.t. the filtration of $Q_n^{(b)}$ and has quadratic covaration $\langle \mathcal{M}\rangle(t) =\int_0^t Q_n^{(b)}(s)\,ds \le Q_n^{(b)}(0) t$.
If $f_n(t) = f_n(0) - \int_0^t f_n(s)\,ds$ then
\begin{equation*}
    Q_n^{(b)}(t) - f_n(t)   =  \left(Q_n^{(b)}(0) - f_n(0)\right) - \mathcal{M}(t) - \int_0^t \left(Q_n^{(b)}(s)-f_n(s)\right)\,ds.
\end{equation*}
Setting $f_n(0) = Q_n^{(b)}(0)$ we get
\begin{equation*}
    \sup_{t\le T} \left| \frac{1}{n} Q_n^{(b)}(t) - \frac{1}{n}f_n(t) \right|\le \sup_{t\le T} \frac{1}{n}|\mathcal{M}(t)| + \int_0^T \sup_{t\le u}  | \frac{1}{n}Q_n^{(b)}(t) - \frac{1}{n}f_n(t) |\,du.
\end{equation*}
Hence Gronwall's lemma implies
\begin{equation*}\sup_{t\le T} \left| \frac{1}{n} Q_n^{(b)}(t) - \frac{1}{n}f_n(t) \right|\le \frac{e^T}{n}\sup_{t\le T} |\mathcal{M}(t)|.
\end{equation*}
By Doob's inequality and $\gamma = \lim_n 2Q_n^{(b)}(0)/n\in (0,\infty)$ by~\eqref{eqn:dBconve}
\begin{equation*}
    \PR(\sup_{t\le T}|\mathcal{M}(t)|> \sqrt{n}a) \le \frac{1}{na^2} \E\left[\langle\mathcal{M}\rangle(T)\right]\le \frac{1}{n a^2} Q_n^{(b)}(0)T \le \frac{\gamma T}{a^2},
\end{equation*} for all large $n$ so that $Q_n(0)/n \le \gamma$. How large $n$ needs to be does not depend on $T$ nor $a$.

Since $f_n(t) = Q_n^{(b)}(0) e^{-t}$, we get for all large $n$ and for all $a$ and all $T$
\begin{align*}
    \PR\left(\sup_{t\le T} \left|\frac{1}{n}Q_n^{(b)} (t) - \frac{1}{n} Q_n^{(b)}(0)e^{-t}\right|> \frac{a}{\sqrt{n}} \right) \le \PR\left(\sup_{t\le T}|\mathcal{M}(t)|\ge \sqrt{n}a e^{-T}\right)\le \frac{\gamma Te^{2T}}{a^2}.
\end{align*} This gives the second bound by letting $a = \sqrt{n} \delta_n$, replacing $T$ with $T/c_n$, and choosing $n$ so that $\exp(2T/c_n) \le2$.

Now if $T_n$ is chosen so that $\frac{1}{n}Q_n^{(b)}(0) (1-e^{-T_n}) \le \eps_n/2$ then
\begin{align*}
    \PR&\left(\sup_{t\le T_n} \left|\frac{1}{n}Q_n^{(b)} (t) - \frac{1}{n} Q_n^{(b)}(0)\right|> \eps_n \right) \\
    &\le
    \PR\left(\sup_{t\le T_n}\left|\frac{1}{n}Q_n^{(b)} (t) - \frac{1}{n} Q_n^{(b)}(0)e^{-t}\right|> \frac{1}{2}\eps_n  \right) + \PR\left(\sup_{t\le T_n}\left|\frac{1}{n}Q_n^{(b)} (0) - \frac{1}{n} Q_n^{(b)}(0)e^{-t}\right|> \frac{1}{2}\eps_n  \right) \\
    &\le \frac{4\gamma T_n e^{2T_n}}{\eps_n^2 n}.
\end{align*}
In particular, this holds for all large $n$ whenever $T_n =\frac{T}{c_n}$ and $\frac{\eps_n}{2}\ge \gamma (1-e^{-T/c_n})$. So setting $\eps_n = \frac{\omega(n)}{c_n}$ for some function $\omega(n)\to\infty$ but $\omega(n) = o(c_n)$. Rewriting the upper bound above for this choice of $\eps_n$ gives for all $n$ large enough so that the above bounds hold and $e^{2T_n}\le 2$ that
\begin{equation*}
    \frac{4 \gamma T_n e^{2T_n}}{\eps_n^2 n} \le 8\gamma T \frac{c_n}{\omega(n)^2 n} \longrightarrow 0,
\end{equation*} as desired. 
\end{proof}

The next lemma tells us that the probability that some vertex in connected component $\cC_n(j)\subset \G_n(0)$ shares an edge with some vertex in $\cC_n(i)\subset \G_n(0)$ at time $s/c_n$ in the unmodified graph process $\G_n(s/c_n)$ is asymptotically the same as in the MCMW. To make the notation easier, given two disjoint sets $A,B$ of vertices, we write $A\sim_{t} B$ if there exists a black edge between a vertex $v\in A$ and a vertex $u\in B$ in $\G_n(t)$. 

\begin{lemma} \label{lem:CONVERG1}Assume Assumption~\ref{ass:Degree} and~\eqref{eqn:dBconve}. Suppose that components $\cC_n^0(j)$ of $\G_n(0)$ are arranged so that $b_n^{-1}\bX_n\weakarrow \bX$ and $b_n^{-1}\bY_n\weakarrow \bY$ in $\ell^2.$
Then for all $i,j$
\begin{equation*}
    \PR\left(\cC_n^0(j) \sim_{\frac{s\gamma }{c_n}} \cC_n^0(i) \Big|\G_n(0)\right) \weakarrow 1-\exp(-Y_iY_j s).
\end{equation*}    

More generally, suppose that $i_0,\dotms, i_k$ are distinct indices.
\begin{equation*}
    \PR\left(\cC_n^0(i_{j-1} )\sim_{\frac{s\gamma}{ c_n}} \cC_n^0(i_j) \,\,\forall j\in[k]  \Big| \G_n(0) \right) \weakarrow \prod_{j=1}^k (1-\exp(-sY_{i_{j-1}}Y_{i_j})).
\end{equation*}
If the convergence $b_n^{-1}\bY\to \bY$ is almost sure the the conclusions above are almost sure.
\end{lemma}
\begin{proof}
By Skorohod's representation theorem, we suppose that $b_n^{-1}\bX_n$ and $b_n^{-1}\bY_n$ converge a.s.  We will also assume that we have two sequences $\delta_n\to 0$ and $\eps_n\to 0$ at the appropriate rate that will be clearly stated later.

Let us now examine how many black-half edges are paired in a connected component for the unmodified graph process ${\G}_n(s/c_n)$. Using the coupling in Lemma~\ref{lem:UpperCoupling}, we can analyze how many edges are created in the modified process $\overline{\G}_n(s/c_n)$ which will be at least as many as in $\G_n(s/c_n)$. For any given collection of vertices $A$ an edge is connected to vertex in $A$ at rate
\begin{align*}
    Q_n^{(b)}(0) &\left( 2\frac{\#H^{(b)}(A)}{2Q_n^{(b)}(0)} \frac{2Q_n^{(b)}(0) - \#H^{(b)}(A)}{2Q_n^{(b)}(0) - 1} + \frac{\#H^{(b)}(A) (\#H^{(b)}(A)-1)}{2Q_n^{(b)}(0)(2Q_n^{(b)}(0)-1)}  \right)\\
    &\le\#H^{(b)}(A).
\end{align*}
The first term above corresponds to pairing a single vertex with a different connected component, and the second corresponds to pairing two vertices in $A$ with each other. Note that each edge added in the modified graph process $\overline{\G}_n(s)$ removes at most two black half-edges from the corresponding collection of vertices $A$ in the unmodified graph process $\G_n(s)$. Let us write $H^{(b)}_s(A)$ for the collection of black half-edges incident to the vertices of $A$ at time $s$ for the graph process $\G_n(s/c_n)$. The above analysis implies that for any collection of vertices $A$, we have
\begin{align}
   \label{eqn:HalfEdgesDeleted1}\PR&\left(\#H_{s/c_n}^{(b)}(A) < (1-\eps_n) \#H^{(b)}_0(A) \right) =\PR\left(\#H^{(b)}_0(A) - \#H^{(b)}_{s/c_n}(A) > \eps_n\#H_{0}^{(b)}(A)\right) \\&\le \PR\left(\textup{Poisson}\left(2\# H^{(b)}_0(A) \frac{s}{c_n}\right) >\eps_n \#H_{0}^{(b)}(A) \right) \le \nonumber \frac{2s}{\eps_nc_n},
\end{align} where the last inequality is Markov's inequality.

Given a collection of vertices $A$, define the stopping time 
\begin{equation*}
    \tau_n^A = \inf\{t: \#H_{t}^{(b)}(A) \le (1-\eps_n) \#H_0^{(b)}(A)\}.
\end{equation*} Given two disjoint collection of vertices let $\tau_n^{AB} = \tau_n^{A}\wedge \tau_n^B$. Observe that on $\{t< \tau_n^{AB}\} \cap \{Q_{n}^{(b)}(t)\le m\}$ we have added at least $Q_n^{(b)}(0)-m$ many black edges into the graph $\G_n(t)$, and an edge between $A$ and $B$ is added each time with probability bounded below by
\begin{equation*}
    2 \frac{ (1-\eps_n)^2H^{(b)}_0(A) H_0^{(b)}(B)}{4 Q_n^{(b)}(0)^2}.
\end{equation*} Moreover, each edge added is independent of the pairing of the white half-edges in the graph $\G_n(0)$. In particular, looking at $t = s/c_n$ and $m = Q_n^{(b)}(0)e^{-s/c_n} + n\delta_n$ 
\begin{align}
\nonumber
    \PR&\left(A\sim_{s/c_n} B\bigg| \G_n(0), \{\frac{s}{c_n}< \tau_n^{AB}\} \cap \{Q_{n}^{(b)}(\frac{s}{c_n})\le Q_n^{(b)}(0) e^{-s/c_n}+n\delta_n\}\right)\\
   \label{eqn:lowerboundForMC1}&\ge  \PR\left(\textup{Bin}\left(Q_n^{(b)}(0)(1-e^{-s/c_n}) - n\delta_n, 2 \frac{ (1-\eps_n)^2H^{(b)}_0(A) H_0^{(b)}(B)}{4 Q_n^{(b)}(0)^2} \right) \ge 1\bigg|\G_n(0)\right).
\end{align}
In order to apply the Binomial to Poisson convergence, we need to carefully choose $\delta_n$ so that $Q_n^{(b)}(0) (1-e^{-s/c_n}) - n\delta_n = \frac{n}{c_n}\frac{s\gamma}{2} (1+o(1))$. This can be done by setting $\delta_n = n^{-\varrho}$ where $\varrho \in (\frac{\tau-3}{\tau-1},\frac{1}{2})$ so that
\begin{equation}\label{eqn:choiceDelta}
\frac{b_n^2 \delta_n}{n}=     c_n \delta_n  = \frac{n^{\frac{\tau-3}{\tau-1}-\varrho}}{L(n)^2} \longrightarrow 0\qquad\textup{and}\qquad n\delta_n^2 = n^{1-2\rho}\to\infty.
\end{equation}In particular, using Lemma~\ref{lem:QnSize} and~\eqref{eqn:HalfEdgesDeleted1} for $\eps_n = \omega(n)/c_n\to 0$ with $\omega(n)\to\infty$ and $\delta_n\to 0$ but $\delta^2_n n\to\infty$ we can bound for all large $n$
\begin{align}
    \nonumber \PR&\left(\{\frac{s}{c_n}< \tau_n^{AB}\} \cap \{Q_{n}^{(b)}(\frac{s}{c_n})\le Q_n^{(b)}(0) e^{-s/c_n}+n\delta_n\}\right)\\
     \nonumber &\ge 1 - \PR\left(\tau_n^{A}\le s/c_n\right) - \PR\left(\tau_n^B \le s/c_n \right) -  \PR\left(Q_{n}^{(b)}(\frac{s}{c_n})> Q_n^{(b)}(0) e^{-s/c_n}+n\delta_n \right)\\
    \nonumber &\ge 1-\frac{ 2s}{\omega(n)} - \frac{2 s}{\omega(n)} - \frac{8 \gamma s }{\delta_n^2 n}  = 1-o(1).
\end{align}

In particular, using Lemma~\ref{lem:QnSize} and~\eqref{eqn:HalfEdgesDeleted1} for $\eps_n = \omega(n)/c_n\to 0$ with $\omega(n)\to\infty$ and $\delta_n\to 0$ but $\delta^2_n n\to\infty$ we can bound for all large $n$
\begin{align}
    \nonumber \PR&\left(\{\frac{s}{c_n}< \tau_n^{AB}\} \cap \{Q_{n}^{(b)}(\frac{s}{c_n})\le Q_n^{(b)}(0) e^{-s/c_n}+n\delta_n\}\right)\\
     \nonumber &\ge 1 - \PR\left(\tau_n^{A}\le s/c_n\right) - \PR\left(\tau_n^B \le s/c_n \right) -  \PR\left(Q_{n}^{(b)}(\frac{s}{c_n})> Q_n^{(b)}(0) e^{-s/c_n}+n\delta_n \right)\\
    \label{eqn:lowerboundMC2} &\ge 1-\frac{ 2s}{\omega(n)} - \frac{2 s}{\omega(n)} - \frac{8 \gamma s }{\delta_n^2 n}  \ge (1-o(1)).
\end{align}

Using~\eqref{eqn:lowerboundMC2} we can re-write the lower bound in~\eqref{eqn:lowerboundForMC1} as
\begin{align*}
    \PR&\left(A\sim_{s/c_n}B|\G_n(0)\right) \\
    &\ge \PR\left(A\sim_{s/c_n} B\bigg| \G_n(0), \{\frac{s}{c_n}< \tau_n^{AB}\} \cap \{Q_{n}^{(b)}(\frac{s}{c_n})\le Q_n^{(b)}(0) e^{-s/c_n}+n\delta_n\}\right)\\
    &\qquad\times \PR\left(\{\frac{s}{c_n}< \tau_n^{AB}\} \cap \{Q_{n}^{(b)}(\frac{s}{c_n})\le Q_n^{(b)}(0) e^{-s/c_n}+n\delta_n\}\right) \\
    &\ge\left(1-o(1)\right) \PR\left(\textup{Bin}\left(Q_n^{(b)}(0)(1-e^{-s/c_n})-n\delta_n, \frac{ (1-\eps_n)^2H^{(b)}_0(A) H_0^{(b)}(B)}{2 Q_n^{(b)}(0)^2} \right)\ge 1\bigg|\G_n(0)\right)
\end{align*}
Now examine the above term with $A = \cC_n^0(i)$ and $B = \cC_n^0(j)$, which are both $\G_n(0)$- measurable. Note that $H_0^{(b)}(A) = Y_{n,i} = b_n (Y_{i}+o(1))$ and $H_0^{(b)}(B) = b_n (Y_j+o(1))$. Also $Q_n^{(b)}(0) = \frac{1}{2}\sum_{i} d_i^{(b)} =  \frac{\gamma}{2}n (1+o(1))$, $1-e^{-s/c_n} = \frac{s}{c_n}(1+o(1))$ and since $b_n^2 = nc_n$, we have by Binomial to Poisson convergence
\begin{align*}
    \PR&\left(\cC_n^0(i)\sim_{s/c_n} \cC_n^0(j)|\G_n(0)\right)\\
&\ge \left(1-o(1)\right) \PR\left(\textup{Bin}\left(\frac{n}{c_n} \frac{\gamma s}{2}(1+o(1)) , \frac{b_n^2}{\frac{n^2}{2} \gamma^2}Y_{i}Y_j (1+o(1))\right)\ge 1 \bigg|\G_n(0)\right)\\
&\longrightarrow 1-\exp\left(-\frac{s}{\gamma} Y_iY_j\right),
\end{align*} whenever $\delta_n$ is chosen as in~\eqref{eqn:choiceDelta}. This establishes the lower bound. The upper bound follows from Lemma~\ref{lem:modified=MC} and the graph coupling in Lemma~\ref{lem:UpperCoupling}.

The more general case proceeds similarly. The main difference is that instead of using a binomial random variable in~\eqref{eqn:lowerboundForMC1}, one uses a multinomial random variable. We omit the details.
\end{proof}

\subsection{Putting Things Together}

We now have all the necessary tools to finish the proof of Theorem~\ref{thm:main2}. 
\begin{proof}
 We will couple $\G_n(s)$ with $\overline{\G}_n(s)$ as in Lemma~\ref{lem:UpperCoupling}. We write $\cC_n(j)$ for the connected components of $\G_n(\mu\gamma c_n^{-1})$ and we write $\cC_n^0(j)$ for the connected components of $\G_n(0)$. We also write $\overline{\cC}_n(j)$ for the connected components of $\overline{\G}_n(\mu\gamma c_n^{-1})$.

Let $\bX_n$ denote the component sizes $(\#\cC_n^0(j);j\ge 1)$ and let $\bY_n$ denote the number of incident black half-edges $(\#H^{(b)}(\cC_n^0(j));j\ge 1)$. After using the Skorohod representation theorem and Theorem~\ref{thm:main1}, we suppose that the components of $\G_n(0)$ are indexed in such a way that $b_n^{-1}\bX_n\to \bX$ and $b_n^{-1}\bY_n\to \bY$ a.s. in $\ell^2$ for this fixed arrangement. 
    
Observe that each connected component of $\G_n(\mu \gamma c_n^{-1})$ is the union of connected components of $\G_n(0)$. More precisely, we have the equivalence relation $\sim_{\mu \gamma/c_n}$ on the connected components $\cC_n^0(j)$ and each connected component $\cC_n(i)$ is the union of sets in an equivalence class. A similar statement holds for the connected components of the larger graph $\overline{\G}_n(\mu\gamma c_n^{-1})$.

We turn this into two graphs as follows. Let $\sG_n = \sG_n(\mu)$ be the graph on $\N$ with $i\sim j$ if and only if $\cC_n^0(i)\sim_{\mu\gamma c_n^{-1}} \cC_n^0(j)$. Similarly, let $\overline{\sG}_n = \overline{\sG}_n(\mu)$ be the graph on $\N$ such that $i\sim j$ if and only if there is an edge a vertex in $\cC_n^{0}(i)$ and a vertex in $\cC_n^{(0)}(j)$ in the modified graph process $\overline{\G}_n(\mu\gamma c_n^{-1})$. We assign mass $X_{n,i}$ to the vertex $i$ of $\sG_n$ and $\overline{\sG}_n$. With this simplification, the ordered vector of component sizes $\ord(\#\cC_n(j);j\ge 1)$ (resp. $\ord(\#\overline{\cC}_n(j);j\ge1)$) of $\G_n(\mu\gamma c_n^{-1})$ (resp. $\overline{\G}_n(\mu\gamma c_n^{-1})$) is equal to the ordered vector of component masses of $\sG_n$ (resp. $\overline{\sG}_n)$). 

Moreover, Lemma~\ref{lem:modified=MC} implies that 
\begin{align*}
    \PR(i\sim j\textup{ in }\overline{\sG}_n) = 1-\exp\left(- b_n^{-1} Y_{n,i} b_n^{-1}Y_{n,j} \mu (1+o(1))\right)
\end{align*} and we can use the $\xi$-coupling to couple the graphs $\overline{\sG}_n$ over the index $n$. With this coupling, we then have
\begin{equation*}
    \overline{\sG}_n\rightsquigarrow \sW^{G}(\bX,\bY,\mu).
\end{equation*} Appealing to Lemma~\ref{lem:CONVERGE2} we see that the ordered vector of component masses, $\overline{\bZ}_n$, of $\overline{\sG}_n$ satisfies
\begin{equation*}
    b_n^{-1} \overline{\bZ}_n\weakarrow \bZ \overset{d}{=} \MC_2(\bX,\bY,\mu).
\end{equation*} where $\bZ$ is the component masses of $\sW^G(\bX,\bY,\mu)$. Finally, $\sG_n\subset \overline{\sG}_n$ and 
\begin{equation*}
    \PR(i\sim j\textup{ in } {\sG}_n) \longrightarrow 1-\exp(-Y_iY_j \mu) = \PR(i\sim j\textup{ in } \sW^G(\bX,\bY,\mu)).
\end{equation*} Hence, Proposition~\ref{prop:graphConvergence} implies the desired result.
\end{proof}

\appendix
\section{Proof of~\eqref{eqn:martingaleArgument} in Lemma~\ref{lem:TailSumsConv}}

Recall $M_n(t) = \frac{1}{a_n}\sum_{i=m+1}^n d_i^{(w)}\left(1_{[\eta_i\le t]} -\frac{d_i^{(w)}t}{\ell_n^{(w)}-2Tb_n} \right)$ and by \cite[pg 1526]{DvdHvLS.20} that $M_n$ is an $\F_l$ super-martingale starting from $M_n(0) = 0$. Using~\eqref{eqn:etaIprobBounds}, we have for all $n$ large and all $t\le Tb_n$
\begin{align*}
    |\E&[M_n(t)]|= -\E[M_n(t) ] = \frac{1}{a_n} \sum_{i=m+1}^n d_i^{(w)}\left(\frac{d_i^{(w)} t}{\ell_n-2Tb_n} - \PR(\eta_i\le t)\right)\\
    &\le \frac{t}{a_n}\sum_{i=m+1}^n d_i^{(w)}\left(\frac{d_i^{(w)} }{\ell_n-2Tb_n} - \frac{d_i^{(w)}}{\ell_n} + \frac{(d_i^{(w)})^2Tb_n}{\ell_n^2} \right)\\
    &\le \frac{Tb_n}{a_n}\sum_{i=m+1}^n d_i^{(w)} \left(\frac{2Tb_n d_i^{(w)}}{\ell_n(\ell_n-2Tb_n)}+ \frac{(d_i^{(w)})^2Tb_n}{\ell_n^2} \right)\\
    &\le \frac{2(Tb_n)^2}{a_n} \sum_{i=m+1}^n \frac{(d_i^{(w)})^2}{\ell_n(\ell_n-2Tb_n)}+ \frac{(d_i^{(w)})^3}{\ell_n^2},
\end{align*}
 where in second to last line last line we used $t\le Tb_n$ and $\frac{1}{x}-\frac{1}{x-a} = \frac{a}{x(x-a)}$.
For all large $n$, $\ell_n - 2Tb_n > \frac{1}{2}\ell_n$, and so we can write
\begin{align*}
    |\E[M_n(t)]|&\le \frac{4(Tb_n)^2}{a_n \ell_n^2} \sum_{i=m+1}^n \left((d_i^{(w)})^2 + (d_i^{(w)})^3\right).
\end{align*}
Using $\ell_n = \Theta(n) = \Theta(a_nb_n)$ under Assumption~\ref{ass:Degree}\eqref{enum:AssDegree3}, we have get $\frac{4T^2 b_n^2}{a_n\ell_n^2} = \Theta(a_n^{-3})$ and so Assumption~\ref{ass:Degree}\eqref{enum:AssDegree3} implies
\begin{align}\label{eqn:EMno1}
    |\E[M_n(t)]|&= \Theta\left({a_n^{-3}}\right)  \sum_{i=m+1}^n (d_i^{(w)})^2+ (d_i^{(w)})^3 = o(1).
\end{align} Indeed, the $(d_i^{(w)})^2$ term is simply Assumption~\ref{ass:Degree}\eqref{enum:AssDegree3}. The $(d_i^{(w)})^2$ sum is also $o(1)$ as $n\to\infty$ because $\sum_{i=m+1}^n (d_i^{(w)})^2 = O(n)$ by Assumption~\ref{ass:Degree}\eqref{enum:AssDegree3} and $a_n^3 = L(n) n^{{3}/({\tau-1})} \gg n$ as $\tau\in(3,4)$. 

Recall~\eqref{eqn:boundingSupdiw}. Observe that by~\eqref{eqn:etaIprobBounds} and Assumption~\ref{ass:Degree}\eqref{enum:AssDegree2} that for all $n$ large and uniformly in $i\in[n]$ and $t\le Tb_n$
\begin{align*}
    \Var&(1_{[\eta_i\le t]})  = \PR(\eta_i\le t) \le \left( \frac{d_i^{(w)}}{\ell_n-2Tb_n} + \frac{(d^{(w)}_i)^2}{(\ell_n-2Tb_n)^2} \right) t\\
    &\le \frac{d_i^{(w)} t}{\ell_n-2Tb_n}  \left(1+\sup_{i} \frac{d_i^{(w)}}{\ell_n-2Tb_n}\right) = \Theta(1)\frac{d_i^{(w)}(Tb_n)}{n}(1+o(1))  \le \Theta(1)\frac{d_i^{(w)}}{a_n}, 
\end{align*} where we used $a_nb_n = n$. We also have, using negative correlation of the Bernoulli random variables $1_{[\eta_i\le b_nt]}$, that for all $t\le Tb_n$
\begin{align}\label{eqn:varMno1}
    \Var(M_n(t)) &\le \sum_{i=m+1}^n \frac{(d_i^{(w)})^2}{a_n^2} \Var(1_{[\eta\le t]})  = \Theta(1)\sum_{i = m+1}^n \frac{(d_i^{(w)})^3}{a_n^3}  =o(1),
\end{align} where the last equality is by Assumption~\ref{ass:Degree}\eqref{enum:AssDegree3}.

Note that for any random variable $X$ with finite variance we have
\begin{equation*}
    \E[|X|] \le \sqrt{\E[X^2]} = \sqrt{\Var(X)+ \E[X]^2} \le \sqrt{\Var(X)} + |\E[X]|.
\end{equation*} In particular, since we have shown $\Var(M_n(t)) = o(1)$ in~\eqref{eqn:varMno1} and $\E[M_n(t)] = o(1)$ in~\eqref{eqn:EMno1} we get
\begin{equation*}
    \E[|M_n(Tb_n)|] = o(1).
\end{equation*}
By Lemma 2.54.5 in \cite{RW.94} we can bound for any $\eps>0$
\begin{equation*}
    \PR \left(\sup_{s\le Tb_n} |M_n(s)|>\eps \right) \le \frac{9}{\eps} \E[|M_n(Tb_n)|]  \longrightarrow 0.
\end{equation*} This implies the stated result.

\bibliography{ref}
\bibliographystyle{amsplain}
\end{document}

%% file: preamble.tex
\usepackage{amsmath,amsthm,amssymb}
\usepackage{mathrsfs}
\usepackage{enumerate}
\usepackage{txfonts}

\usepackage{bbm}
\usepackage[square,sort, compress,numbers]{natbib}
\usepackage[colorlinks,linkcolor=blue,citecolor=blue,urlcolor=blue]{hyperref}
\usepackage{latexsym}
\usepackage{color,graphicx}
\usepackage{bm} 
\usepackage{float}
\usepackage{tikz}
\usepackage{pdfpages}
\usepackage{cases}

\allowdisplaybreaks

\newcommand{\CM}{{\operatorname{CM}}}
\newcommand{\HCM}{{\operatorname{HCM}}}
\newcommand{\bd}{\boldsymbol{d}}

\newcommand{\bzeta}{\boldsymbol{\zeta}}
\newcommand{\bbeta}{\boldsymbol{\beta}}
\newcommand{\MC}{{\operatorname{MC}}}
\newcommand{\ord}{{\operatorname{ord}}}
\newcommand{\sW}{\mathscr{W}}
\newcommand{\sB}{\mathscr{B}}

\newcommand{\sM}{{\mathscr{M}}}
\newcommand{\sG}{{\mathscr{G}}}

 \newcommand{\bY}{\mathbf{Y}}
\tikzstyle{new edge style 0}=[line width=1.5pt]
\newcommand{\bz}{{\mathbf{z}}}
\newcommand{\bZ}{{\mathbf{Z}}}

\makeatletter
\def\namedlabel#1#2{\begingroup
    #2%
    \def\@currentlabel{#2}%
    \phantomsection\label{#1}\endgroup
}
\makeatother

\newcommand{\bS}{{\mathbf{S}}}

\newtheorem{theorem}{Theorem}[section]
\newtheorem{corollary}[theorem]{Corollary}
\newtheorem{lemma}[theorem]{Lemma}
\newtheorem{proposition}[theorem]{Proposition}

\newtheorem{claim}[theorem]{Claim}

\theoremstyle{definition}
\newtheorem{definition}[theorem]{Definition}
\newtheorem{algorithm}[theorem]{Algorithm}
\newtheorem{assumption}[theorem]{Assumption}

\DeclareMathOperator{\Var}{Var}

\DeclareMathAlphabet{\mathscrbf}{OMS}{mdugm}{b}{n}

\numberwithin{equation}{section}

\newcommand{\dotms}{\dotsm}

\newcommand{\Leb}{\operatorname{Leb}}

\newcommand{\by}{{\bf y}}
\newcommand{\bx}{{\bf x}}

\newcommand{\bw}{{\bf w}}

\newcommand{\eps}{\varepsilon}
\newcommand{\R}{\mathbb{R}}

\newcommand{\D}{{\mathbb{D}}}

\newcommand{\N}{\mathbb{N}}

\newcommand{\sC}{{\mathscr{C}}}

\newcommand{\F}{\mathcal{F}}

\newcommand{\E}{\mathbb{E}}

\newcommand{\cY}{\mathcal{Y}}
\newcommand{\bX}{{\bf X}}

\newcommand{\weakarrow}{{\overset{(d)}{\Longrightarrow}}}

\newcommand{\PR}{\mathbb{P}}

\newcommand{\G}{{\mathcal{G}}}

\newcommand{\cC}{{\mathcal{C}}}

\newcommand{\EE}{\mathcal{E}}
\newcommand{\LL}{\mathcal{L}}

%% file: main.bbl
\providecommand{\bysame}{\leavevmode\hbox to3em{\hrulefill}\thinspace}
\providecommand{\MR}{\relax\ifhmode\unskip\space\fi MR }
\providecommand{\MRhref}[2]{%
  \href{http://www.ams.org/mathscinet-getitem?mr=#1}{#2}
}
\providecommand{\href}[2]{#2}
\begin{thebibliography}{10}

\bibitem{Aldous.97}
David Aldous, \emph{Brownian excursions, critical random graphs and the multiplicative coalescent}, Ann. Probab. \textbf{25} (1997), no.~2, 812--854. \MR{1434128}

\bibitem{AL.98}
David Aldous and Vlada Limic, \emph{The entrance boundary of the multiplicative coalescent}, Electron. J. Probab. \textbf{3} (1998), no. 3, 59. \MR{1491528}

\bibitem{BBSW.14}
Shankar {Bhamidi}, Nicolas {Broutin}, Sanchayan {Sen}, and Xuan {Wang}, \emph{{Scaling limits of random graph models at criticality: Universality and the basin of attraction of the Erd{\H{o}}s-R{\'e}nyi random graph}}, arXiv e-prints (2014), arXiv:1411.3417.

\bibitem{BDvdHS.20}
Shankar Bhamidi, Souvik Dhara, Remco van~der Hofstad, and Sanchayan Sen, \emph{Universality for critical heavy-tailed network models: metric structure of maximal components}, Electron. J. Probab. \textbf{25} (2020), Paper No. 47, 57. \MR{4092766}

\bibitem{BD.14}
Ginestra Bianconi and Sergey Dorogovtsev, \emph{Multiple percolation transitions in a configuration model of a network of networks}, Physical Review E \textbf{89} (2014), no.~6, 062814.

\bibitem{Billingsley.99}
Patrick Billingsley, \emph{{Convergence of Probability Measures}}, second ed., Wiley Series in Probability and Statistics: Probability and Statistics, John Wiley \& Sons, Inc., New York, 1999, A Wiley-Interscience Publication. \MR{1700749}

\bibitem{BDW.22}
Nicolas Broutin, Thomas Duquesne, and Minmin Wang, \emph{Limits of multiplicative inhomogeneous random graphs and l{\'e}vy trees: The continuum graphs}, The Annals of Applied Probability \textbf{32} (2022), no.~4, 2448--2503.

\bibitem{CKL.22}
{David Clancy, Jr., Vitalii Konorovskyi, and Vlada Limic}, \emph{{Excursion representation for the degree-corrected stochastic blockmodel}}, {In preperation}.

\bibitem{DvdHvLS.20}
Souvik Dhara, Remco van~der Hofstad, Johan S.~H. van Leeuwaarden, and Sanchayan Sen, \emph{Heavy-tailed configuration models at criticality}, Ann. Inst. Henri Poincar\'{e} Probab. Stat. \textbf{56} (2020), no.~3, 1515--1558. \MR{4116701}

\bibitem{DvdHvLS.17}
Souvik Dhara, Remco van~der Hofstad, Johan~S.H. van Leeuwaarden, and Sanchayan Sen, \emph{Critical window for the configuration model: finite third moment degrees}, Electronic Journal of Probability \textbf{22} (2017), 1--33.

\bibitem{JS.13}
Jean Jacod and Albert Shiryaev, \emph{{Limit Theorems for Stochastic Processes}}, vol. 288, Springer Science \& Business Media, 2013.

\bibitem{Janson.10}
Svante Janson, \emph{Susceptibility of random graphs with given vertex degrees}, J. Comb. \textbf{1} (2010), no.~3-4, 357--387. \MR{2799217}

\bibitem{Joseph.14}
Adrien Joseph, \emph{The component sizes of a critical random graph with given degree sequence}, Ann. Appl. Probab. \textbf{24} (2014), no.~6, 2560--2594. \MR{3262511}

\bibitem{KL.21}
Vitalii Konarovskyi and Vlada Limic, \emph{Stochastic block model in a new critical regime and the interacting multiplicative coalescent}, Electron. J. Probab. \textbf{26} (2021), Paper No. 30, 23. \MR{4235481}

\bibitem{Kryven.19}
Ivan Kryven, \emph{Bond percolation in coloured and multiplex networks}, Nature communications \textbf{10} (2019), no.~1, 404.

\bibitem{RW.94}
L.~C.~G. Rogers and David Williams, \emph{Diffusions, {M}arkov processes, and martingales. {V}ol. 1}, second ed., Wiley Series in Probability and Mathematical Statistics: Probability and Mathematical Statistics, John Wiley \& Sons, Ltd., Chichester, 1994, Foundations. \MR{1331599}

\bibitem{SvdHvL.16a}
Clara Stegehuis, Remco van~der Hofstad, and Johan~S.H. van Leeuwaarden, \emph{Epidemic spreading on complex networks with community structures}, Scientific reports \textbf{6} (2016), no.~1, 29748.

\bibitem{SvdHvL.16}
\bysame, \emph{Power-law relations in random networks with communities}, Physical Review E \textbf{94} (2016), no.~1, 012302.

\bibitem{vdHvLS.18}
Remco van~der Hofstad, Johan~S.H. van Leeuwaarden, and Clara Stegehuis, \emph{Mesoscopic scales in hierarchical configuration models}, Stochastic processes and their applications \textbf{128} (2018), no.~12, 4246--4276.

\bibitem{Warnke.19}
Lutz Warnke, \emph{{On Wormald's differential equation method}}, arXiv preprint arXiv:1905.08928 (2019).

\bibitem{Whitt.80}
Ward Whitt, \emph{Some useful functions for functional limit theorems}, Mathematics of Operations Research \textbf{5} (1980), no.~1, 67--85.

\bibitem{Wormald.99}
Nicholas~C Wormald, \emph{Differential equations for random processes and random graphs}, The Annals of Applied Probability (1995), 1217--1235.

\end{thebibliography}
